\documentclass[11pt]{article}
\usepackage[a4paper]{geometry}
\usepackage{amsthm}
\usepackage{amsmath}
\usepackage{amssymb}
\usepackage{epsfig,graphicx}
\usepackage{url}
\usepackage{subfigure}

\newtheorem{Lemma}{Lemma}
\newtheorem{Proposition}{Proposition}
\newtheorem{Corollary}{Corollary}
\newtheorem{Theorem}{Theorem}

\theoremstyle{remark}
\newtheorem*{Remark}{Remark}

\begin{document}
\title{Hyperbolic periodic orbits in nongradient systems and small-noise-induced metastable transitions}

\author{Molei Tao}

\maketitle

\abstract{Small noise can induce rare transitions between metastable states, which can be characterized by Maximum Likelihood Paths (MLPs). Nongradient systems contrast gradient systems in that MLP does not have to cross the separatrix at a saddle point, but instead possibly at a point on a hyperbolic periodic orbit. A numerical approach for identifying such unstable periodic orbits is proposed based on String method. In a special class of nongradient systems (`orthogonal-type'), there are provably local MLPs that cross such saddle point or hyperbolic periodic orbit, and the separatrix crossing location determines the associated local maximum of transition rate. In general cases, however, the separatrix crossing may not determine a unique local maximum of the rate, as we numerically observed a counter-example in a sheared 2D-space Allen-Cahn SPDE. It is a reasonable conjecture that there are always local MLPs associated with each attractor on the separatrix, such as saddle point or hyperbolic periodic orbit; our numerical experiments did not disprove so.}

% keywords:	metastable transition; nongradient systems; Freidlin-Wentzell large deviation; hyperbolic periodic orbit; String method

\section{Introduction and main results}
Rare dynamical events induced by small noise can nevertheless be important. Examples of reactive rare events include climate changes, phase transitions, and switching of macromolecular conformations \cite{HeVa08c}. It is not ideal to study these events by integrating the dynamics, because most of the computation will be wasted on waiting for rare events to occur. Freidlin-Wentzell large deviation theory \cite{FrWe12} provides an assessment of likelihoods of rare events. More precisely, consider an SDE
\begin{equation}
    dX=f(X)dt + \sqrt{\epsilon} dW,
    \label{eq_generalSDE}
\end{equation}
where $X\in \mathbb{R}^d$, $\epsilon$ is a small parameter, and $W$ is a $d$-dimensional Wiener process\footnote{The assumption of additive noise is not essential and only for simplicity; see for instance \cite{HeVa08a} on generalized situations.}. Freidlin-Wentzell theory states, as $\epsilon\rightarrow 0$ and given boundary condition $X(T_1)=x_a$ and $X(T_2)=x_b$, the probability density of a solution $X(\cdot)$ is asymptotically proportional to $\exp(-S_{T_1,T_2}[X]/\epsilon)$, where the action functional is given by
\begin{equation}
    S_{T_1,T_2}[X]:=\begin{cases}
        \frac{1}{2}\int_{T_1}^{T_2} \left\| \dot{X}(s)-f(X(s)) \right\|^2 ds,	& \quad X\in \bar{\mathcal{C}}_{x_a}^{x_b}(T_1,T_2) \\
        \infty,	& \quad X \not\in \bar{\mathcal{C}}_{x_a}^{x_b}(T_1,T_2)
    \end{cases},
    \label{eq_FWaction}
\end{equation}
where $\bar{\mathcal{C}}_{x_a}^{x_b}(T_1,T_2)$ is the space of absolutely continuous functions in $[T_1,T_2]$ that satisfy $X(T_1)=x_a$ and $X(T_2)=x_b$. 

In the $\epsilon \rightarrow 0$ limit, the transition probability is characterized by the minimizer of the action. In addition, in many situations $T_1,T_2$ is not known, and in this case it is natural to perform an additional outer minimization over all $T_1<T_2$. If one does so, the minimum is generally achieved when $T_2-T_1\rightarrow \infty$ \cite{ren2004minimum, HeVa08a}. Therefore, from now on, we assume $T_1=-\infty, T_2=\infty$ and seek minimizers\footnote{In most parts of this article we will only seek local minimizers. The reason is convexity is not guaranteed and global minimization might be too difficult.} of $S_\infty[\cdot]$ with boundary conditions. Such minimizers will be called maximum likelihood paths (MLPs) throughout this article (they are also called instantons in the physical literature). Also, we will be mainly working with metastable transitions, i.e., $x_a$ and $x_b$ are two stable fixed points in the noise-less system $\dot{X}=f(X)$.

If the system is gradient, i.e., there is a scalar field $V(\cdot)$ such that $f=-\nabla V$, it is known that a MLP between two local minima of $V$ coincides with a Minimum Energy Path (MEP), which is defined as a trajectory along which $-\nabla V$ is always parallel to the path. Many computational methods have been developed to compute MEPs, such as \cite{ulitsky1990new,fischer1992conjugate,olender1996calculation,berne1998classical,henkelman1999dimer,henkelman2000climbing,henkelman2000improved,StringMethod,StringSimplified}. Among them is String method \cite{StringMethod,StringSimplified}, which is compared with others in \cite{sheppard2008optimization}, briefly summarized in Appendix \ref{sec_review_String}, and will be modified in Section \ref{sec_pString}.
In addition, it is known that an MEP has to cross at least one saddle point of the potential energy $V$, which corresponds to a saddle point on a separatrix submanifold in the noise-less dynamical system (e.g., \cite{FrWe12}). It was further shown that the identification of this saddle point helps improve MEP computations (e.g., \cite{sheppard2008optimization}). A number of approaches have been proposed to locate saddle points, including \cite{crippen1971minimization, henkelman1999dimer, weinan2011gentlest, zhang2012shrinking, ren2013climbing}.

It should be no surprise that transitions in nongradient systems can be more complicated. After all, gradient systems correspond to nonequilibrium statistical mechanics that are reversible diffusion processes (i.e., satisfying detailed balance), while nongradient systems may correspond to irreversibility (e.g., see \cite{FrWe12,maier1993escape}). Since many important systems are nongradient, including Langevin models of mechanical systems in constant temperature environment (e.g., \cite{gardiner1985handbook}), stochastic fluid models (e.g., \cite{flandoli1995martingale}), or irreversible coarse-grained systems (e.g., \cite{maier1992transition}), numerous efforts have been made to understand metastable transitions in nongradient systems. These include \cite{Bouchet2016, landim2016metastability, maier1996oscillatory, berglund2004noise, berglund2014noise, berglund2016noise, maier1993escape, maier1997limiting, kraut2004escaping, silchenko2005fluctuational, kautz1988thermally, cameron2012finding, zhou2010study, wan2010study, newhall2013averaged, newhall2016metastability, berglund2013kramers}, which will be discussed in the context of this article after three paragraphs. One key issue with general nongradient systems is, MEPs are no longer defined, because there is no energy landscape on which the system evolves, and there may be no path whose tangent aligns with $f$ everywhere. However, MLPs as minimizers of the action functional \eqref{eq_FWaction} can still be investigated, and their identification is essential for characterizing rare events in these systems. Several successful numerical methods for computing MLPs have been proposed, and we refer to \cite{ren2004minimum, HeVa08a, VaHe08b, HeVa08c, zhou2008adaptive} for examples. Among them is geometric Minimum Action Method (gMAM) \cite{HeVa08a, VaHe08b, HeVa08c}, which is briefly summarized in Appendix \ref{sec_review_gMAM} and will be modified later in this article.

As a general nongradient theory is still incomplete, this article makes a small step by showing the followings: unlike in gradient systems, MLP does not have to cross a saddle point in a nongradient system, and in fact there may be no saddle point at all. The second simplest limit set, namely periodic orbit, which is generally excluded in gradient systems, may be present on the separatrix and utilized by the metastable transition. More specifically, for a class of nongradient systems dubbed `orthogonal-type', the transition rate is again characterized by a barrier height like in the gradient case (this result was stated in \cite{FrWe12}), and given a saddle point or a hyperbolic periodic orbit that locally attracts on the separatrix, there is a unique associated local minimum action and an explicitly defined path that achieves this action. Interestingly, periodic-orbit-crossing MLPs differ significantly from a saddle-crossing MLP, as their arc-lengths are infinite and there are infinitely many of them, even if there is only one periodic orbit. On the other hand, for a non-orthogonal-type nongradient system, numerically obtained local MLPs also cross saddles or periodic orbits, but there can be multiple local MLPs that correspond to the same separatrix crossing location but with different action values.

Two numerical methods play critical roles in this study. One of them identifies hyperbolic periodic orbits, based on a variation of String method \cite{StringMethod,StringSimplified}. The other numerically computes MLPs by supplementing the geometric Minimum Action Method (gMAM, \cite{HeVa08a, VaHe08b, HeVa08c}) with information about the separatrix crossing locations.

Several facts have to be mentioned: (i) There have been previous studies on transitions through periodic orbits. Most of these studies considered an unstable periodic orbit (rather than hyperbolic), which is the boundary of the attraction basin of a metastable state, and the systems are naturally 2D (e.g., \cite{maier1996oscillatory, berglund2004noise, berglund2014noise}). There is a study of the 3D Lorenz system, based on careful numerical investigations \cite{zhou2010study}. In addition, a recent work \cite{berglund2016noise} considered systems in which periodic orbits can be globally characterized by phase angle variables, and demonstrated metastable transitions between two stable periodic orbits through unstable periodic orbits. This article focuses on hyperbolic periodic orbits for less specific problems and the dimension can be much higher. (ii) This article is based on the traditional Freidlin-Wentzell large deviation theory and thus does not discuss the prefactor of the transition rate. However, several important contributions have been made to analyze nongradient systems \cite{maier1993escape, Bouchet2016, landim2016metastability}, and they quantitatively discussed the prefactor given by the Eyring-Kramer formula (see also \cite{berglund2013kramers} for a review). (iii) Most theoretical claims in this article are natural consequences of Freidlin and Wentzell's results on orthogonal-type nongradient systems (see Chap 4.3 of \cite{FrWe12}), and our purpose is to combine them with numerical investigations to make the link between periodic orbits and rare events explicit. A beautiful concurrent article \cite{Bouchet2016} also studied nongradient systems using the same tool of orthogonal decomposition (along with other powerful machinery such as asymptotic analysis), but its scope is complementary, because it assumed saddle point is the only attractor on the separatrix. Note the orthogonal-type system considered here was called in that article (a system admitting) `transverse decomposition'. (iv) Nongradient systems in 2D have been extensively investigated (e.g., \cite{kautz1988thermally,maier1997limiting,cameron2012finding} in addition to aforementioned \cite{maier1996oscillatory, berglund2004noise, berglund2014noise}), several high-dimensional systems of practical relevance have been explored \cite{newhall2013averaged, wan2010study, newhall2016metastability}, and discrete systems have also been studied (e.g., \cite{kraut2004escaping,silchenko2005fluctuational}).

This article is organized as follows. Section \ref{sec_orthoSystem} analyzes orthogonal-type nongradient systems so that the link between hyperbolic periodic orbit and MLP can be explicitly established. Section \ref{sec_examples} uses concrete examples to illustrate several features of metastable transitions distinct from gradient systems. % We show that String method may not extend to Purposes of this description include (i) to show that String method generally does not extend to a nongradient system; (ii) to demonstrate analytically that (local) MLP between metastable states can go through saddle-like periodic orbits instead of saddle points; (iii) to illustrate some interesting features of MLPs that cross periodic orbits, such as infinite length and non-uniqueness, which also affect their numerical computations.
In Section \ref{sec_pString}, String method is modified to identify hyperbolic periodic orbits in deterministic dynamical systems. Section \ref{sec_results} demonstrates how this identification improves gMAM-based MLP computation; using this improved numerical tool, phase space structures of a (2+1)-dimensional reaction-diffusion-advection PDE are explored, and differences between orthogonal- and non-orthogonal-type systems are discussed.

Many examples in this article are generalizations of a 1D gradient system with double well potential $V(x)=(1-x^2)^2/4$, but the specific form of this potential is not essential --- similar conclusions will apply to smooth bistable potentials. However, our investigation is limited to systems with two stable fixed points. In principle, it is possible to study systems with more sinks by first investigating each barrier crossing event using similar techniques and then constructing a network of barrier crossings (see for example \cite{FrWe12,freidlin2000quasi,bovier2004metastability,bovier2005metastability,schutte2011markov}), but it is beyond the scope of this article.

\section{Nongradient systems of orthogonal-type}
\label{sec_orthoSystem}
\subsection{The orthogonal-type system}
\label{sec_orthoSystem_system}
Consider a class of nongradient systems defined on $\mathbb{R}^d$, in the form of
\begin{equation}
    dX= ( -\nabla V(X)+b(X) )dt + \sqrt{\epsilon} dW,
    \label{eq_orthoSystem}
\end{equation}
where $\nabla V$ and $b$ are smooth and satisfy $\nabla V(x)\cdot b(x)=0$ for all $x\in \mathbb{R}^d$. Consider also the deterministic version
\begin{equation}
    \dot{X}= -\nabla V(X)+b(X).
    \label{eq_orthoSystem_deterministic}
\end{equation}
Suppose \eqref{eq_orthoSystem_deterministic} contains two stable fixed points $x_a$ and $x_b$, and their basins of attractions $\mathcal{D}_a$ and $\mathcal{D}_b$ cover the entire phase space (i.e., $\overline{\mathcal{D}_a} \bigcup \overline{\mathcal{D}_b} = \mathbb{R}^d$). Consider the separatrix submanifold $\mathcal{S}$, which is the boundary between basins of attractions of $x_a$ and $x_b$ (i.e., $\mathcal{S}=\partial \mathcal{D}_a \bigcap \partial \mathcal{D}_b$). Assume there is at least one saddle point $x_s$ or hyperbolic periodic orbit $x_{PO}(t)$ in $\mathcal{S}$ that is attracting on $\mathcal{S}$, i.e., with its stable manifold containing a neighborhood of $x_s$ or $\{x_{PO}(t)|\forall t\}$ in $\mathcal{S}$. Assume there is a heteroclinic orbit that goes from $x_a$ to $x_s$ or a point on $x_{PO}(t)$ in an auxiliary dynamical system $\dot{X}= \nabla V(X)+b(X)$.

Fixed point or periodic orbit in system \eqref{eq_orthoSystem_deterministic} satisfies the following:

\begin{Lemma}
    $\nabla V(x_s)=0$.
\end{Lemma}
\begin{proof}
    Since $x_s$ is a fixed point, $-\nabla V(x_s)+b(x_s)=0$. By orthogonality of $\nabla V$ and $b$, both are zero.
\end{proof}

\begin{Lemma}
    $\nabla V(x_{PO}(t))=0$ for all $t$.
    \label{thm_vanishing_dV}
\end{Lemma}
\begin{proof}
    Denote by $T$ the period. Consider $v(t)=V(x_{PO}(t))$. We have
    \[
        \dot{v}=\nabla V \cdot \dot{x}_{PO} = \nabla V \cdot (-\nabla V(x_{PO})+b(x_{PO})) = - \| \nabla V \|^2 \leq 0
    \]
    Since $v(t)=v(t+T)$, $\dot{v}=0$ for all $t$, and thus $\nabla V=0$.
\end{proof}

Moreover, attractions of $x_s$ and $x_{PO}$ lead to:
\begin{Lemma}
    $x_s$ is a local minimum of $V(x|x\in\mathcal{S})$.
\end{Lemma}
\begin{proof}
    By attraction assumption, there exists $r>0$ such that any point $x$ in $B(x_s,r) \cap \mathcal{S}$ approaches $x_s$ in system \eqref{eq_orthoSystem_deterministic}. As shown in the above proof, $V$ is a Lyapunov function of the dynamics, and therefore $V(x)\geq V(x_s)$. Hence $x_s$ is local minimum of $V$ restricted to $\mathcal{S}$.
\end{proof}
\begin{Lemma}
    For any $t$, $x_{PO}(t)$ is a local minimum of $V(x|x\in\mathcal{S})$.
    \label{thm_localMinimizingV}
\end{Lemma}
\begin{proof}
    Analogous to the above saddle point case.
\end{proof}

\subsection{The maximum likelihood transition}
\label{sec_orthogonal_MLP}
In orthogonal-type systems and conditioned on metastable transitions, a local minimizer of the Freidlin-Wentzell action functional can be explicitly obtained, and the corresponding action value is determined by the separatrix crossing location.

More specifically, view fixed point as a degenerate periodic orbit for simplicity, and consider two systems, respectively called uphill and downhill dynamics, given by
\begin{align*}
    &\dot{X}^*_1=+\nabla V(X^*_1)+b(X^*_1), \qquad X^*_1(-\infty)=x_a, \quad X^*_2(\infty)=x_s, \\
    &\dot{X}^*_2=-\nabla V(X^*_2)+b(X^*_2), \qquad X^*_2(-\infty)=x_s, \quad X^*_2(\infty)=x_b,
\end{align*}
where $x_s$ is an arbitrary point in the hyperbolic periodic orbit $\{ x_{PO}(t) \big| t\in \mathbb{R} \}$ (note Lemma \ref{thm_vanishing_dV} guarantees $\{ x_{PO}(t) | t\in \mathbb{R} \}$ is a periodic orbit of both the uphill and the downhill dynamics). The formal usage of boundary conditions at $\pm\infty$ means that $X^*_1$ and $X^*_2$ are heteroclinic orbits in respective systems. The uphill heteroclinic orbit was assumed to exist in Section \ref{sec_orthoSystem_system}. The downhill heteroclinic orbit exists because $x_s$ is in the separatrix $\mathcal{S}$, which is the boundary of the attraction basin of $x_b$ in \eqref{eq_orthoSystem_deterministic}.

Natural consequences of Freidlin and Wentzell's results (Chap 3 in \cite{FrWe12}) are: (i) the concatenation of these two heteroclinic orbits will give in state space the geometric configuration of an action local minimizer, and (ii) once this configuration is known, a local MLP can be obtained by reconstructing the time parameterization of the path, via the fulfillment of $\|X'\|=\|f(X)\|=\sqrt{\|\nabla V(X)\|^2+\|b(X)\|^2}$.

\medskip
To make these claims more precise, let's first recall the definition of quasipotential:
\[
	U(x_a,x_b):=\inf_{T_1,T_2} \inf_{X\in \bar{\mathcal{C}}_{x_a}^{x_b}(T_1,T_2)} S_{T_1,T_2}[X].
\]
A great observation was made in \cite{HeVa08a, VaHe08b, HeVa08c} (see also page 102 in \cite{FrWe12}) that the quasipotential problem can be converted to an equivalent but simpler problem that focuses on the geometry of the minimizer:
\begin{Lemma}[Geometric minimum action]
	\begin{equation}
		U(x_a,x_b)=\inf_{X\in \bar{\mathcal{C}}_{x_a}^{x_b}(0,1)} \hat{S} [X],
    	\label{eq_gMAMminimization}
	\end{equation}
	where the geometric action is defined as
	\[
	    \hat{S}_{T_1,T_2}[X]=\int_{T_1}^{T_2} \big( \|X'\|\|f(X)\|-\langle X',f(X) \rangle \big) ds,
	\]
	and $\hat{S}$ is the short hand for $\hat{S}_{0,1}$.
	\label{thm_geometricAction}
\end{Lemma}
\begin{Remark}
A recap of the main rationale is the following. It is easy to see $S_{T_1,T_2}[X] \geq \hat{S}_{T_1,T_2}[X]$ by Cauchy-Schwarz, but in fact one further has $\inf_{T_1,T_2,X} S_{T_1,T_2}[X]=\inf_{T_1,T_2,X} \hat{S}_{T_1,T_2}[X]$, because time can always be rescaled such that $\|X'(s)\|=\|f(X)\|$, and then $\frac{1}{2}\|\dot{X}(s)-f(X(s)) \|^2=\|X'\|\|f(X)\|-\langle X',f(X) \rangle$. Moreover, it can be seen by chain rule that $\hat{S}_{T_1,T_2}[X]$ does not depend on the time parametrization of the path or detailed values of $T_1,T_2$, and hence $\inf_{T_1,T_2,X} \hat{S}_{T_1,T_2}[X]= \inf_X \hat{S}_{0,1}[X]$.
\end{Remark}

\begin{Theorem}
    Given a point $x_s$ in a hyperbolic periodic orbit that is attracting on the separatrix $\mathcal{S}$, there is an associated local MLP with action value $2(V(x_s)-V(x_a))$, and it corresponds to the concatenation of $X_1^*$ and $X_2^*$.
    \label{thm_action}
\end{Theorem}

\begin{proof}
	Since any path connecting $x_a$ with $x_b$ must cross the separatrix $\mathcal{S}$, Lemma \ref{thm_geometricAction} and that the geometric action $\hat{S}$ is invariant under reparameterization lead to
	\begin{align*}
		&U(x_a,x_b)=\inf_{X\in \bar{\mathcal{C}}_{x_a}^{x_b}(0,1)} \hat{S} [X]
		=\inf_{x_c \in \mathcal{S}, X\in \bar{\mathcal{C}}_(0,1): X(0)=x_a, X(\frac{1}{2})=x_c, X(1)=x_b} \hat{S} [X] \\
		&\quad =\inf_{x_c \in \mathcal{S}} \left( \inf_{X_1\in \bar{\mathcal{C}}_(0,\frac{1}{2}): X_1(0)=x_a, X_1(\frac{1}{2})=x_c} \hat{S}_{0,\frac{1}{2}} [X_1] + \inf_{X_2\in \bar{\mathcal{C}}_(\frac{1}{2},1): X_2(\frac{1}{2})=x_c, X_2(1)=x_b} \hat{S}_{\frac{1}{2},1} [X_2] \right) \\
		&\quad =\inf_{x_c \in \mathcal{S}} \big(U(x_a,x_c)+U(x_c,x_b)\big).
	\end{align*}
	Therefore, it suffices to show that $x_c=x_s$ is a local minimum of $U(x_a,x_c)+U(x_c,x_b)$.

    It is easy to see that $U(x_s,x_b)=0$, because $X_2^*$ by definition is the zero (and hence the minimizer) of
    \[
        \frac{1}{2} \int_{-\infty}^\infty \|\dot{X}-( -\nabla V(X)+b(X) )\|^2 ds.
    \]
    Furthermore, it was proved in \cite{FrWe12} page 100 that the same action functional under the constraints of $X(-\infty)=x_a$ and $X(\infty)=x_c$ was bounded from below by $2(V(x_c)-V(x_a))$, and that $X_1^*$ corresponds to the action value of $2(V(x_s)-V(x_a))$.
    
    Since $\nabla V(x_s)=0$ due to Lemma \ref{thm_vanishing_dV}, $x_s$ is a critical point of $V$. Moreover, Lemma \ref{thm_localMinimizingV} shows $x_s$ is a local minimum of $V(x_c \big| x_c\in\mathcal{S})$. Therefore, there is an open ball $\mathcal{B}(x_s,\eta)$ such that $\inf_{x_c \in \mathcal{B}(x_s,\eta) \cap \mathcal{S} } U(x_a,x_c) = 2(V(x_s)-V(x_a))$, and the infimum is attained by $X_1^*$.

    Therefore, $X_1^*$ and $X_2^*$ together attain a local minimum of $U(x_a,x_c)+U(x_c,x_b)$ where $x_c\in \mathcal{S}$ is the variable.
\end{proof}

\begin{Remark}
    Different choices of $x_s \in \{x_{PO}(t) \big| t\in\mathbb{R} \}$ on the periodic orbit lead to the same action value (although not the same path), because Lemma \ref{thm_vanishing_dV} ensures that $V$ is constant along the periodic orbit.
\end{Remark}

\begin{Remark}
	The operation of `concatenation' could be made precise. For instance, one could let $\hat{X}:[0,1]\rightarrow\mathbb{R}^d$ be
	\[
		\hat{X}(t)=\begin{cases} 
			X_1^*(\tan \left(2\pi t-\frac{\pi}{2})\right), \qquad & t\in \left[ 0,\frac{1}{2} \right] \\
			X_2^*(\tan \left(2\pi t-\frac{3\pi}{2})\right), \qquad & t\in \left[ \frac{1}{2},1 \right]
		\end{cases}.
	\]
	$\hat{X}$ will be a local minimizer of $\hat{S}$. A path $X^*:(-\infty,\infty)\rightarrow \mathbb{R}^d$ that locally minimizes the original action $S$ can then be obtained via $X^*(\tau)=\hat{X}(t)$, where $\frac{dt}{d\tau}=\frac{\|f(\hat{X})\|}{\|\hat{X}'\|}$. \end{Remark}

\subsection{Numerical challenges in computing the heteroclinic orbits}
\label{sec_orthoSystem_numerics}
Assuming the separatrix crossing location is known, mathematically it suffices to find the uphill and downhill heteroclinic orbits $X_1^*$ and $X_2^*$ for identifying the local MLP. However, the numerical computations of $X_1^*$ and $X_2^*$ are nontrivial.

Downhill dynamics is relatively easy to obtain. If one manages to obtain a perturbation $x^+$ of $x_s$ such that $x^+\in \mathcal{D}_b^\circ$ and $\|x^+-x_s\| < \epsilon$ for some small $\epsilon$, then downhill orbit can be well approximated by the initial value problem
\[
    \dot{X_2}=-\nabla V(X_2)+b(X_2), \qquad X_2(-T)=x^+.
\]
for some large $T$. This is because $x_b$ is an attractor in this dynamical system.

Uphill dynamics can be more difficult to obtain. It is possible (see for example Section \ref{sec_AC1D}) that $-\nabla V$ has a nonzero projection onto the stable subspace at the separatrix crossing location, corresponding to a stable direction of the periodic orbit in the original dynamics. In this situation, the periodic orbit will become unstable in the uphill dynamical system $\dot{X_1}=\nabla V(X_1)+b(X_1)$ restricted to the separatrix. Consequently, an accurate identification of the heteroclinic orbit between $x_a$ and the periodic orbit in the uphill system is nontrivial due to instability.

The numerical approximation of an unstable heteroclinic orbit is part of an active research field (e.g., \cite{FrDo91,BeKl97,BrOsVe97,DeDiFr00,DiRe04,KrOs05,HuZo12}). The approach we employ will be based on gMAM (Section \ref{sec_results_gMAMextend} and Appendix \ref{sec_review_gMAM}). gMAM is advantageous in our context of transition through periodic orbit, because it is based on action minimization and thus good at ignoring parts with large arclengths that contribute little to the action.

\subsection{Generality of orthogonal-type systems}
To write $\dot{X}=f(X)$ as an orthogonal-type system, it is easy to see $V$ needs to satisfy a PDE
\begin{equation}
    (f(x)+\nabla V(x)) \cdot \nabla V(x) = 0.
    \label{eq_decompositionPDE}
\end{equation}
Depending on $f(\cdot)$, this PDE may admit only trivial solution or multiple nontrivial solutions. For example, when $f(q,p)=(p,-q)$ (i.e., Hamiltonian system of harmonic oscillator), the only solution is $V \equiv \text{constant}$ (see Appendix \ref{sec_orthoHamiltonian} for a proof); however, when $f(x_1,x_2)=(-x_1,-x_2)$, $V(x_1,x_2)$ can at least be $V=\text{constant}$, $V=x_1^2/2+x_2^2/2$, or $V=(x_1-x_2)^2/4$. Analysis and numerical solution of this PDE in 2D are discussed in depth in \cite{cameron2012finding}.

Nevertheless, not all decompositions satisfy assumptions in Section \ref{sec_orthoSystem_system}. In particular, one can always decompose an arbitrary $f(x)$ by picking $b(x)=f(x), V(x)=0$. However, this trivial decomposition will not satisfy the assumption on the existence of an uphill heteroclinic orbit in $\dot{X}=\nabla V(X)+b(X)$.

Besides, \eqref{eq_decompositionPDE} may not be easy to solve in high dimensional cases. For instance, we were not able to verify via \eqref{eq_decompositionPDE} if the example in Section \ref{sec_AC2D_intro} is of orthogonal-type (note there $x$ is in an infinite dimensional function space); instead, we will employ an indirect approach to obtain numerical evidence that it is not (Section \ref{sec_AC2D}).

\section{Transitions in nongradient systems: case studies}
\label{sec_examples}
Many experiences and tools for gradient systems do not generalize to nongradient ones without adaptation. For instance, String method in its classical form \cite{StringMethod,StringSimplified}, which is designed and works beautifully for gradient systems, does not directly apply to examples in this section; nevertheless, it can be adapted to provide critical information about nongradient metastable transitions (Section \ref{sec_pString}). Let us first illustrate several features of metastable transitions not seen in gradient-systems.
\subsection{2D SDE system}
\label{sec_examples_2D}
This is an example for which String method in its classical form \cite{StringMethod,StringSimplified} converges, but does not produce the MLP.

\paragraph{The system.}
Consider a simple system with no periodic orbit:
\begin{equation}
\begin{cases}
    d r= (1-z^2-r) dt + \sqrt{\epsilon} dW_1 \\
    d z= (z-z^3) dt + \sqrt{\epsilon} dW_2
\end{cases}
\label{eq_system2D}
\end{equation}
When noise is absent ($\epsilon=0$), there are 3 fixed points:
$r=0,z=-1$: sink;
$r=0,z=1$: sink;
$r=1,z=0$: saddle.
Also, the separatrix is $z=0$.

This system is nongradient. On the other hand, it is of orthogonal-type, and
\[
    V(r,z)=(1-z^2)^2/4, \qquad b(r,z)=(1-z^2-r,0)
\]
satisfies all assumptions in Section \ref{sec_orthoSystem_system}.

\paragraph{The transition.}
When applied to nongradient systems with general form $dx=f(x) dt+\sqrt{\epsilon} dW$, String method seeks path $\phi(\alpha)$ that satisfies $\phi_\alpha \parallel f(\phi)$. Figure \ref{fig_2Dstring} illustrates the path from $(-1,0)$ to $(1,0)$ obtained by String method.

This path, however, does not correspond to the minimizer of the action (i.e., maximizer of transition rate). Figure \ref{fig_2DgMAM} illustrates the MLP numerically obtained by gMAM. By Theorem \ref{thm_action}, the exact minimum action is $2(V(1,0)-V(-1,0))=0.5$, and gMAM result is an accurate numerical approximation.

\begin{figure}[h]
\centering
\footnotesize
\subfigure[Path obtained by String method.]{
\includegraphics[width=0.48\textwidth]{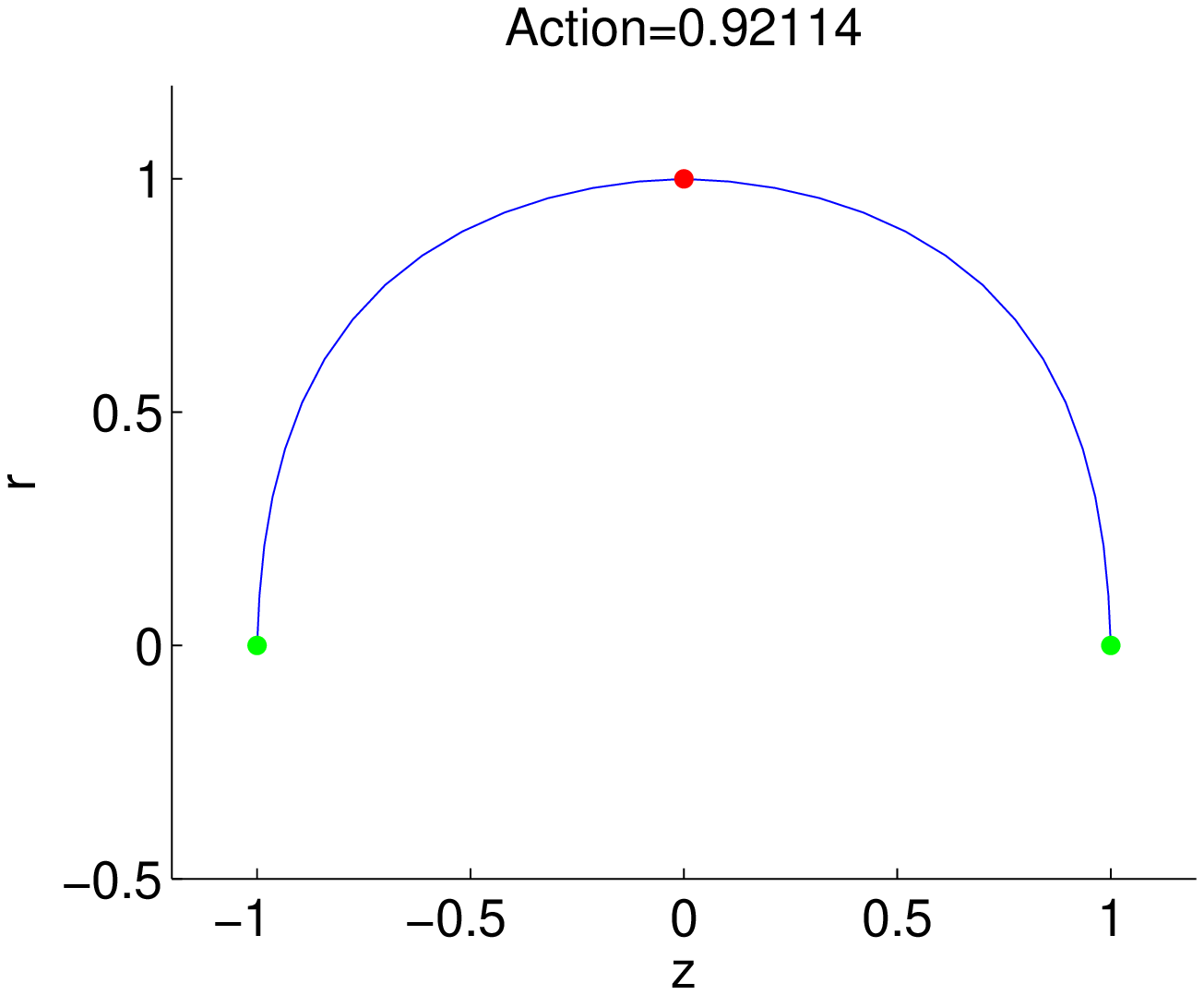}
\label{fig_2Dstring}
}
\subfigure[MLP obtained by gMAM.]{
\includegraphics[width=0.48\textwidth]{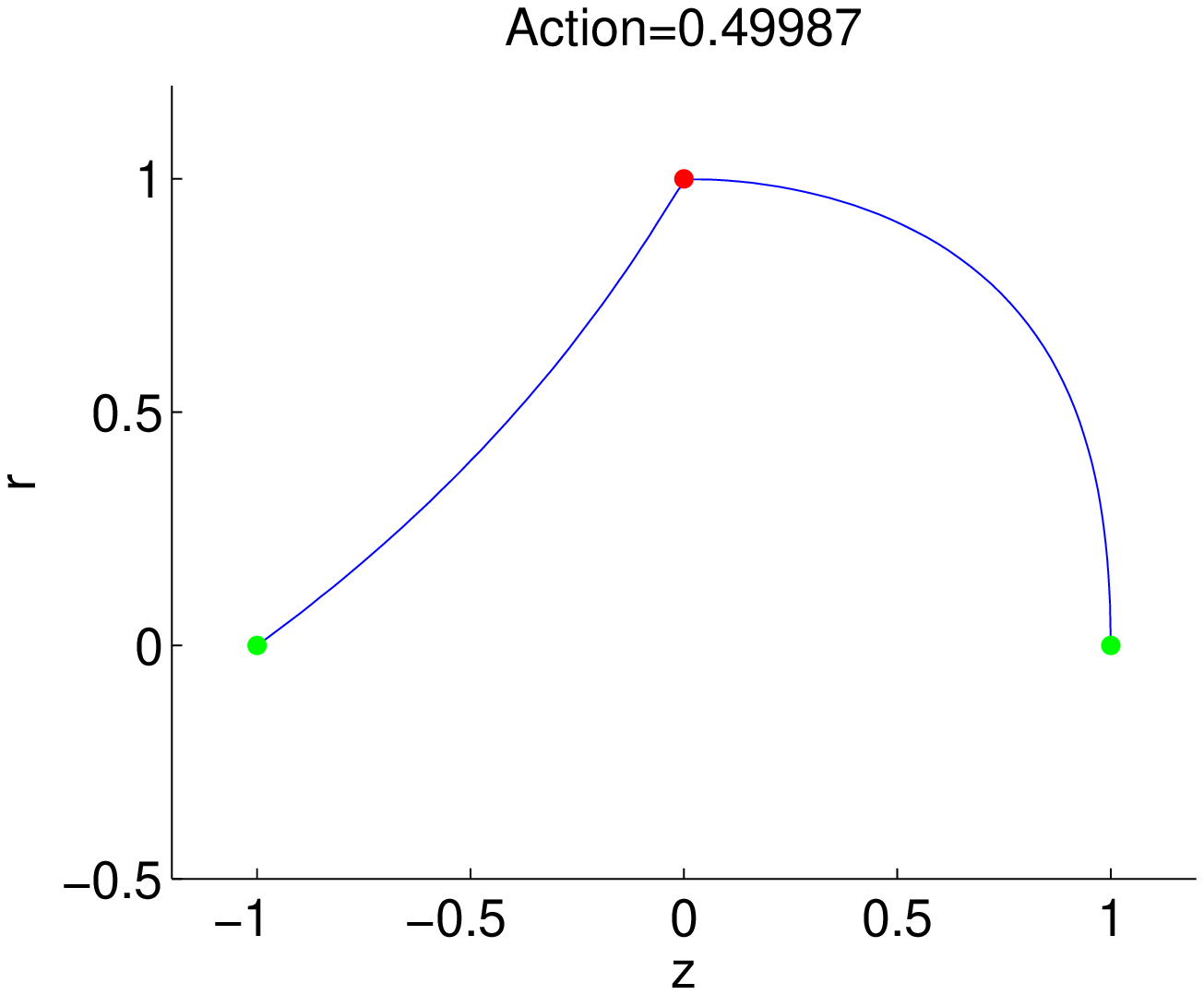}
\label{fig_2DgMAM}
}
\caption{\footnotesize Transitions from (-1,0) to (1,0) in 2D SDE system \eqref{eq_system2D}.}
\end{figure}

%For this special example, the uphill heteroclinic orbit can also be approximated by running String method from $r=0,z=-1$ to $r=1,z=0$ on system
%\[
%\begin{cases}
%    \dot{r}=1-z^2-r \\
%    \dot{z}=z^3-z
%\end{cases} .
%\]
%Same uphill dynamics will be produced, but at a computational cost lower than gMAM. String method works here because $-\nabla V$ only accounts for unstable directions at the saddle, and therefore the uphill dynamics driven by $\nabla V+b$ is stable (see Section \ref{sec_orthoSystem_numerics}).
%
%Also, for this example, the downhill heteroclinic orbit can be approximated by running String method from $r=1,z=0$ to $r=0,z=1$ on the original system. This is advantageous because then one needs not to identify a perturbation $z^+ \in \mathcal{D}_b$.

\subsection{3D SDE system (rotationally-symmetric)}
This is an example for which String method in its classical form \cite{StringMethod,StringSimplified} does not converge. The system contains no saddle point, and an MLP provably crosses a periodic orbit (provable because the system is of orthogonal-type). There are infinitely many MLPs and each is of an infinite length.
\label{sec_examples_3D}
\paragraph{The system.}
Consider
\begin{equation}
\begin{cases}
    dx = \left( (1-z^2) \frac{x}{\sqrt{x^2+y^2}} - x - y \right) dt + \sqrt{\epsilon} dW_1 \\
    dy = \left( (1-z^2) \frac{y}{\sqrt{x^2+y^2}} - y + x \right) dt + \sqrt{\epsilon} dW_2 \\
    dz = \left( z-z^3 \right) dt + \sqrt{\epsilon} dW_3
\end{cases}
\label{eq_system3D}
\end{equation}
When noise is absent ($\epsilon=0$), the system can be rewritten in cylindrical coordinates:
\begin{equation}
\begin{cases}
    \dot{r}=1-z^2-r \\
    \dot{\theta}=1 \\
    \dot{z}=z-z^3
\end{cases}
\label{eq_system3Dcylindrical}
\end{equation}
and recognized as a rotated version of the 2D system \eqref{eq_system2D} without noise. It contains the following limit sets:
\begin{itemize}
\item
    $x=0,y=0,z=-1$: attracting fixed point.
\item
    $x=0,y=0,z=1$: attracting fixed point.
\item
    $z=0,x^2+y^2=1$: periodic orbit, on which $\dot{\theta}=1$; it is saddle-like (i.e. hyperbolic) because it is unstable along $z$ direction but stable in $z=0$ plane.
\end{itemize}
Again, its separatrix is $z=0$. The system is nongradient, of orthogonal-type, with
\begin{align*}
    & V(x,y,z) = (1-z^2)^2/4 \\
    & b(x,y,z) = \left( (1-z^2) \frac{x}{\sqrt{x^2+y^2}} - x - y, ~ (1-z^2) \frac{y}{\sqrt{x^2+y^2}} - y + x, ~ 0 \right).
\end{align*}

\paragraph{MLPs.}
There is a periodic orbit on the separatrix but no fixed point, and this clearly contrasts with gradient systems which do not have periodic orbits. By Theorem \ref{thm_action}, an MLP in this system is the concatenation of two heteroclinic orbits connecting stable points and the periodic orbit. Because the $\dot{\theta}=1$ rotation decouples with the $r, z$ dynamics (see eq. \ref{eq_system3Dcylindrical}), given the MLP $r(t),z(t)$ in the previous 2D example (Section \ref{sec_examples_2D}), a MLP in this system is given by
\[
    x(t)=r(t)\cos(t+\theta_0),\quad y(t)=r(t)\sin(t+\theta_0), \quad z(t) ,
\]
where $\theta_0$ is an arbitrary constant. Because of $\theta_0$, MLPs are not unique.

\begin{figure}[h]
\centering
\footnotesize
\subfigure{
\includegraphics[width=0.48\textwidth]{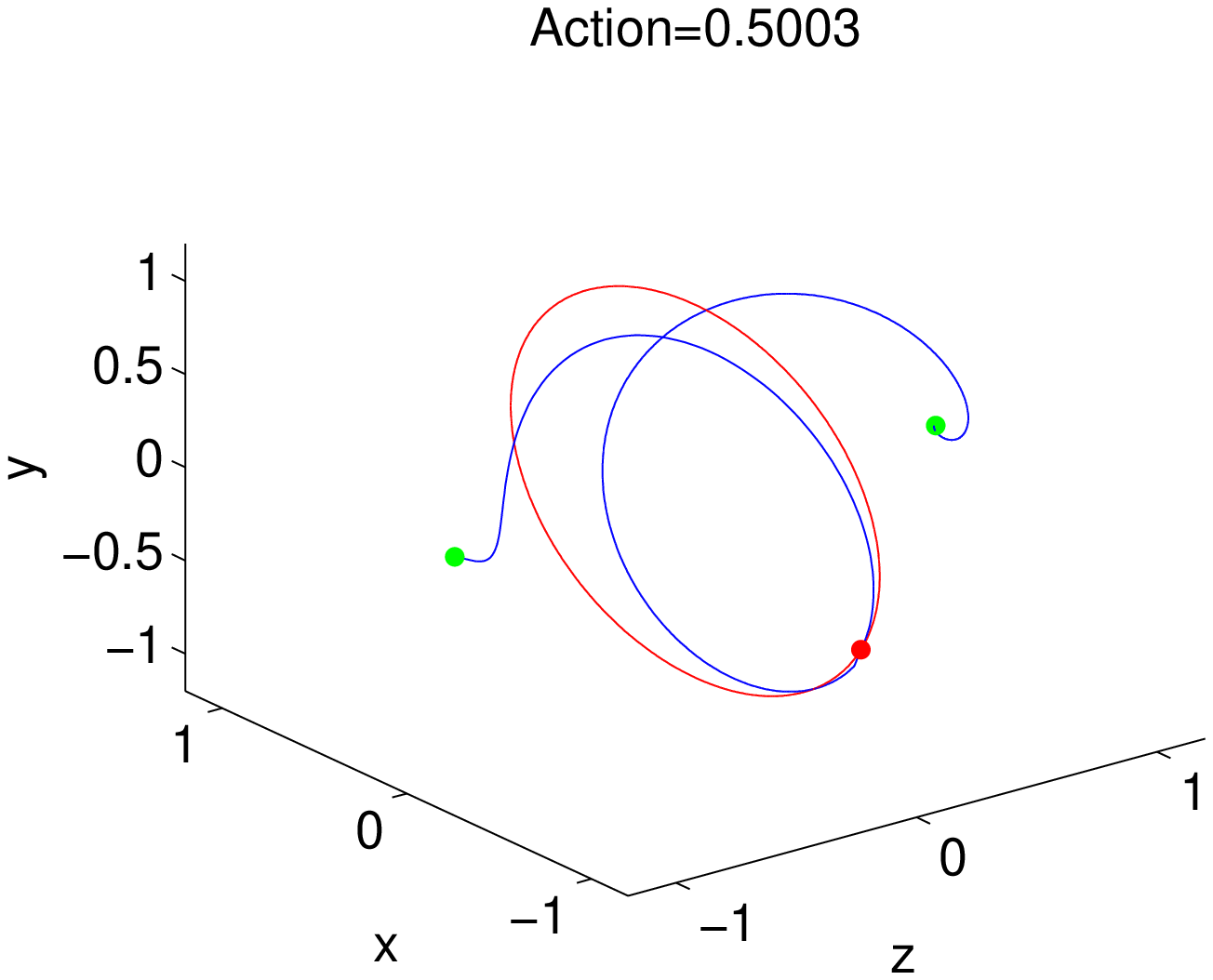}
}
\subfigure{
\includegraphics[width=0.48\textwidth]{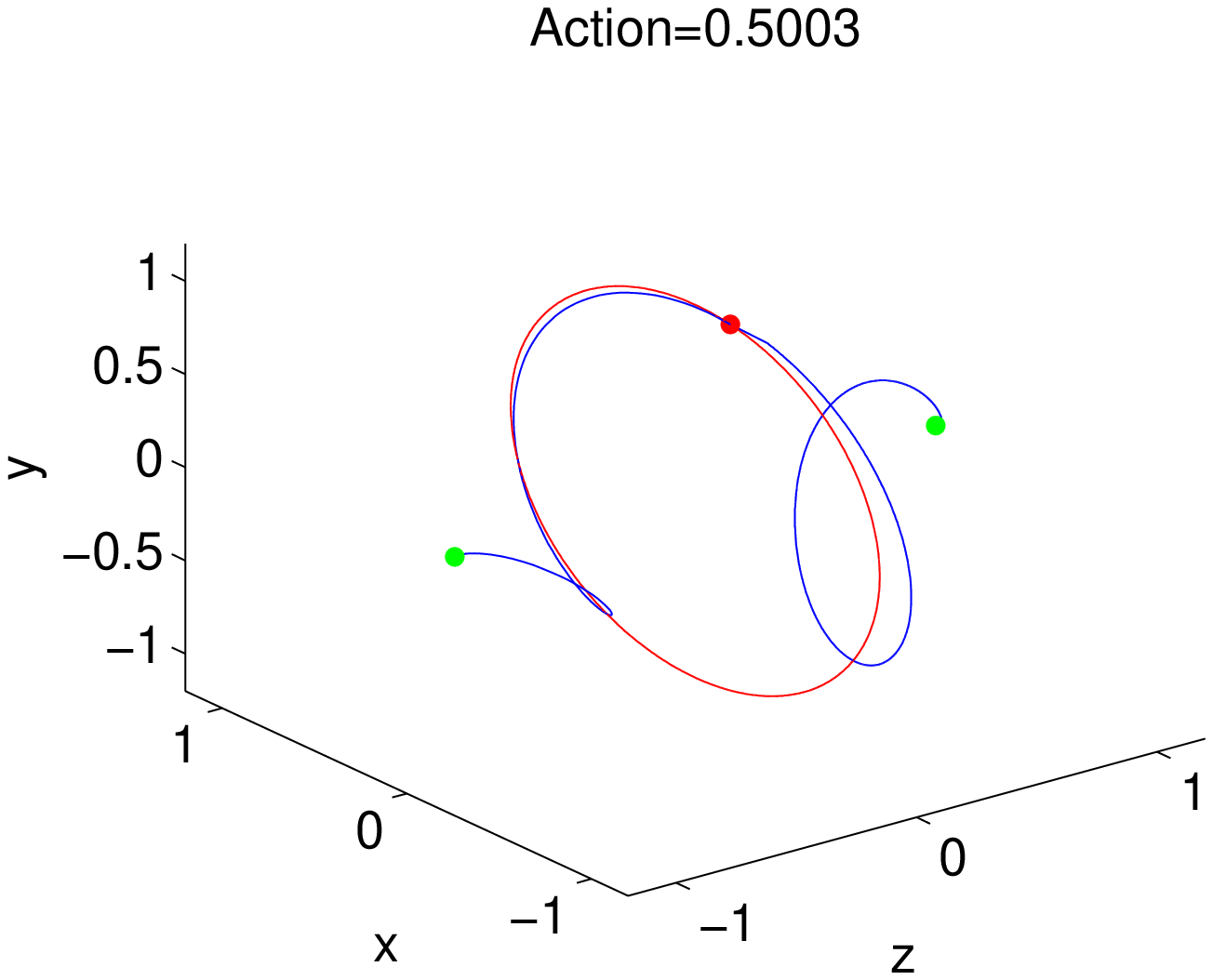}
}
\caption{\footnotesize Finite length approximations of MLPs between (-1,0,0) and (1,0,0) in 3D SDE system \eqref{eq_system3D}. Computed by up-down gMAM (described in Section \ref{sec_results_gMAMextend}) with different intersections at the separatrix.}
\label{fig_3DgMAM}
% old: n1=n2=100, h=0.01, dt=0.1, dynamics threshold=dt/10
% now: n1=100, n2=20, dt=h=0.01, dynamics threshold=dt
\end{figure}

Note the length of this path is infinite. This is because uphill/downhill dynamics take infinite time to reach/leave the separatrix ($z=0$). Since angular velocity is nonzero constant, infinite rotations occur. Because the rotation radius is approximately $r=1$ near the separatrix, infinite rotations lead to an infinite arc-length.

Such an infinite winding can cause numerical problems. In fact, one may be attempted to use String method to obtain the downhill dynamics, but this will not work well because the string will wind more and more around the limit cycle, and any finite discretization of the String will eventually become insufficient and start consuming arc-lengths from the parts away from the separatrix.

gMAM, on the other hand, suits to find a reasonable finite-length approximation of an MLP. This is because the infinite winding near the separatrix takes significant physical time, but contributes little to the action. Since gMAM minimizes a geometrized action that is independent of the time parametrization of the path, when minimizing in a space of finitely-discretized paths, most of the infinite winding can be approximated by a finite-length segment without increasing the action too much.

Figure \ref{fig_3DgMAM} illustrates two MLPs approximated by a variant of gMAM (Section \ref{sec_results}). The exact minimum action is $2(V(\cos\theta,\sin\theta,0)-V(0,0,-1))=0.5$, due to Theorem \ref{thm_action}. The numerical MLPs are, of course, only finite length approximations.

\subsection{3D SDE system (no rotational symmetry)}
This example drops the rotational symmetry of the previous example but remains orthogonal-type. There is still a hyperbolic periodic orbit and no saddle. MLPs again utilize the periodic-orbit and share features similar to the previous example.
\label{sec_examples_3D2}
\paragraph{The system.}
Consider
\begin{equation}
\begin{cases}
    dx = \left( -(z+1)(z-2)\frac{x}{(x^4+y^4)^{1/4}}-x-y^3 \right) dt + \sqrt{\epsilon} dW_1 \\
    dy = \left( -(z+1)(z-2)\frac{y}{(x^4+y^4)^{1/4}}+x^3-y \right) dt + \sqrt{\epsilon} dW_2 \\
    dz = \left( -(z+1)(z-2)z \right) dt + \sqrt{\epsilon} dW_3
\end{cases}
\label{eq_system3D2}
\end{equation}
When noise is absent ($\epsilon=0$), the system contains the following limit sets:
\begin{itemize}
\item
    $x=0,y=0,z=-1$: attracting fixed point.
\item
    $x=0,y=0,z=2$: attracting fixed point.
\item
    $z=0, x^4+y^4=16$: periodic orbit, on which $\dot{x}=-y^3$ and $\dot{y}=x^3$; it is saddle-like because it is unstable along $z$ direction but stable in $z=0$ plane.
\end{itemize}
The separatrix is $z=0$. This nongradient system is again of orthogonal-type, with
\begin{align*}
    & V(x,y,z) = z^4/4-z^3/3-z^2 \\
    & b(x,y,z) = \left( -(z+1)(z-2)\frac{x}{(x^4+y^4)^{1/4}}-x-y^3, ~ -(z+1)(z-2)\frac{y}{(x^4+y^4)^{1/4}}+x^3-y, ~ 0 \right).
\end{align*}

\paragraph{MLPs.}
An MLP is again the concatenation of two heteroclinic orbits. Although there is no longer a rotational symmetry, MLP is still not unique. See Figure \ref{fig_3D2gMAM} for two numerically approximated MLPs. Note the exact minimum action is $2(V(2,0,0)-V(0,0,-1)) = 5/6 \approx 0.8333$, and the gMAM approximations can be improved by using more discretization points. Again, a true MLP is of infinite length due to infinite winding near the periodic orbit.

\begin{figure}[h]
\centering
\footnotesize
\subfigure{
\includegraphics[width=0.48\textwidth]{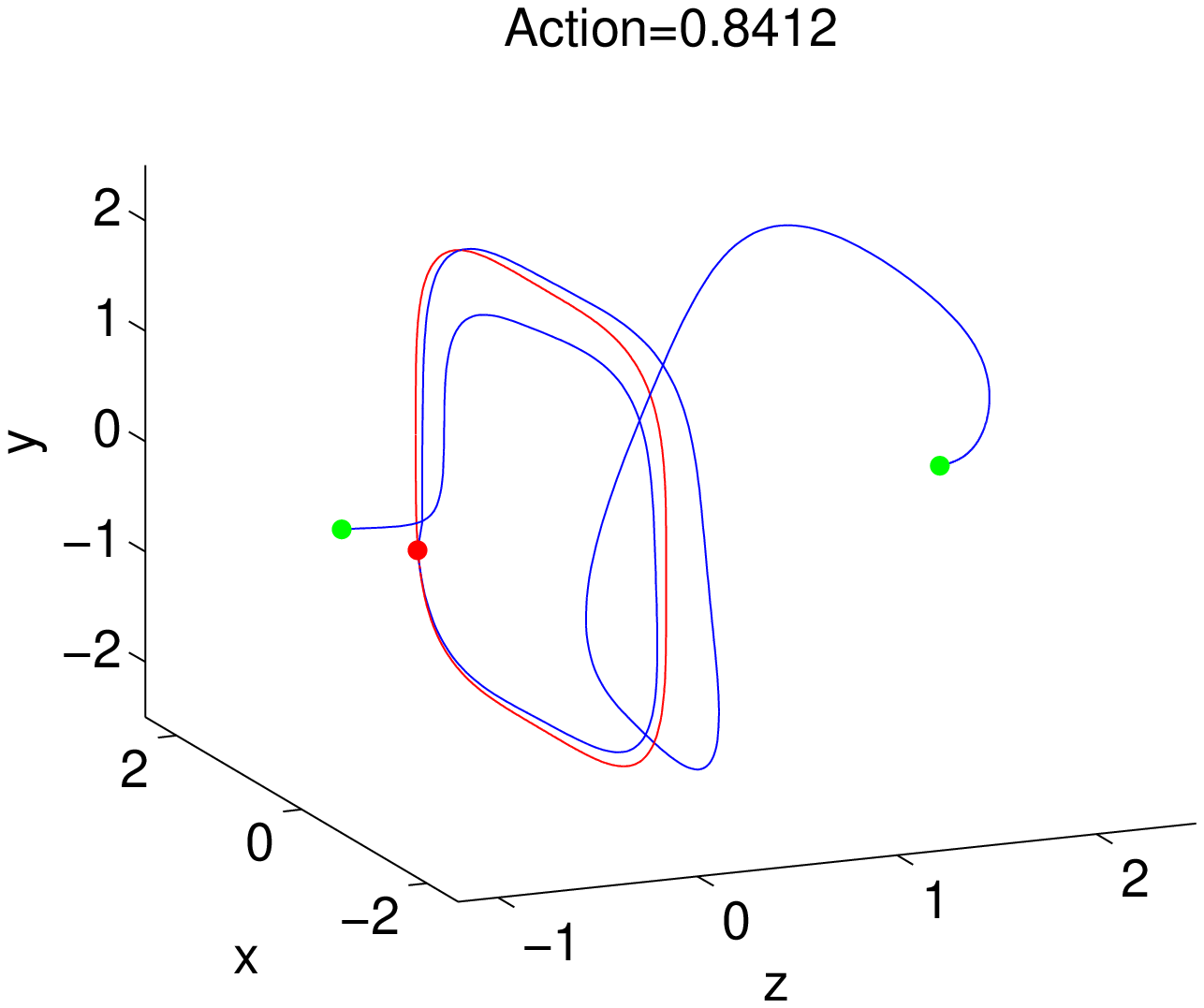}
}
\subfigure{
\includegraphics[width=0.48\textwidth]{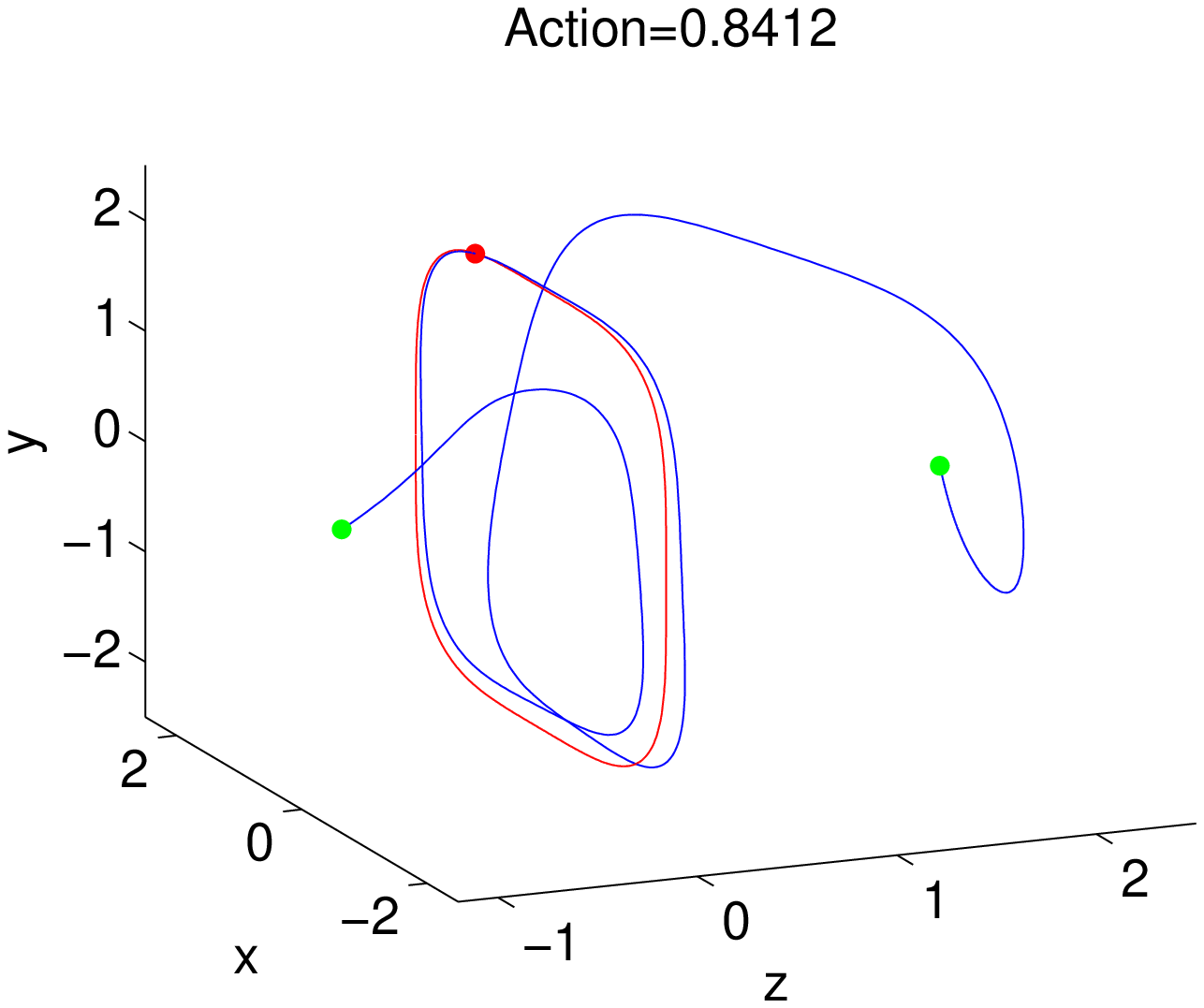}
}
\caption{\footnotesize Finite length approximations of MLPs between (-1,0,0) and (1,0,0) in 3D SDE system \eqref{eq_system3D}. Computed by up-down gMAM (described in Section \ref{sec_results_gMAMextend}) with different intersections at the separatrix.}
% n1=200, n2=40, dt=h=0.01, dynamics threshold=dt
\label{fig_3D2gMAM}
\end{figure}

\subsection{1D-space advection-diffusion-reaction SPDE}
\label{sec_AC1D}
Now consider an infinite-dimensional example. It is nongradient, but still of orthogonal-type, and previously observed transition features persist.
\paragraph{The system.}
Consider
\begin{equation}
    \phi_t=\kappa \phi_{xx}+\phi-\phi^3+c \phi_x +\sqrt{\epsilon}\eta
    \label{eq_AC1D}
\end{equation}
with periodic boundary condition $\phi(x+1,t)=\phi(x,t)$, where $\eta(x,t)$ is spatiotemporal white-noise with covariance $\mathbb{E} [\eta(x,t)\eta(x',t')]=\delta(x-x')\delta(t-t')$; $0<\kappa \ll 1$.

Without advection and noise (i.e., $c=0$, $\epsilon=0$), this system is 1D-space Allen-Cahn \cite{AlCa79}, which is a classical model for alloy. For $c\neq 0$, the advection makes the system nongradient. When viewed as an infinite dimensional dynamical system, \eqref{eq_AC1D} is of orthogonal-type:

\begin{Proposition}
    Given $\mathcal{C}^2$ function $u(x)$ satisfying $u(x)=u(x+1)$, define
    \begin{align*}
        & V[u]=\int_0^1 \kappa \frac{1}{2} u_x^2 + \frac{1}{4}(1-u^2)^2 \, dx ,\\
        & b[u]=c u_x ,
    \end{align*}
    and introduce inner product
    \[
        \left\langle u, v\right\rangle = \int_0^1 u(x)v(x) dx.
    \]
    Then $\left\langle \frac{\delta V}{\delta u}[u], b[u] \right\rangle =0$, and the system \eqref{eq_AC1D} with $\epsilon=0$ is equivalent to
    \[
        \phi_t(\cdot,t)=-\frac{\delta V}{\delta u}[\phi(\cdot,t)]+b[\phi(\cdot,t)].
    \]
    \label{thm_AC1D_orthogonal}
\end{Proposition}
Fixed points and periodic orbits in the $\epsilon=0$ system can also be characterized:
\begin{Proposition}
    When $c\neq 0$, $\epsilon=0$, the dynamical system \eqref{eq_AC1D} contains only three fixed points: $u_s(x) \equiv 0$, unstable; $u_+(x) \equiv 1$, stable; $u_+(x) \equiv -1$: stable.\label{thm_AC1D_fixedPt1}
\end{Proposition}
 
\begin{Proposition}
    When $c=0$, $\epsilon=0$, the dynamical system \eqref{eq_AC1D} contains three uniform fixed points, $u_s(x) \equiv 0$, unstable; $u_+(x) \equiv 1$, stable; $u_+(x) \equiv -1$: stable. Furthermore, when $0<\kappa \leq 1/(2\pi)^2$, there are also finitely many non-uniform fixed points, i.e., non-constant $u(x)$ that solves $\kappa u_{xx}+u-u^3=0$. The number of fixed points nonstrictly increases as $\kappa$ decreases.\label{thm_AC1D_fixedPt2}
\end{Proposition}

\begin{Proposition}
    When $c\neq 0$, $\epsilon=0$, each non-uniform fixed point $u(x)$ in the dynamical system \eqref{eq_AC1D} with $c=0$ bifurcates into a periodic orbit $\phi(x,t)=u(x+ct)$.\label{thm_AC1D_PO}
\end{Proposition}
Proofs are in Appendix \ref{sec_SPDE1D_structures}.

Note we exhausted all fixed points in \eqref{eq_AC1D}, but we have not proved whether periodic orbits identified by Proposition \ref{thm_AC1D_PO} are the only periodic orbits. Worth mentioning is, $\omega$-limit sets other than fixed points and periodic orbits have been ruled out in this specific system \cite{FiMa89}.

Analytical characterization of separatrix structures in this infinite dimensional system is not easy. For some classical results, including investigations on dynamics on the separatrix, we refer to an incomplete list \cite{FiMc77,Ma79,FuHa89,FiMa89,An88,BrFi89,CaPe90,Ch92,DeSc95}.

\paragraph{MLPs:} Consider transitions from $u=-1$ to $u=1$ with $c\neq 0$. The only fixed point on the separatrix is $u=0$. When $\kappa \leq 1/(2\pi)^2$, there are also periodic orbit(s) on the separatrix. Figure \ref{fig_AC1D_gMAM} illustrates three local MLPs\footnote{Sometimes also referred to as nucleation instantons.}, which respectively cross the separatrix at the $u=0$ fixed point, a periodic orbit that bifurcated from a non-uniform fixed point corresponding to a 1-periodic solution of $\kappa u_{xx}+u-u^3=0$, and another periodic orbit that bifurcated from a $\frac{1}{2}$-periodic solution of $\kappa u_{xx}+u-u^3=0$ (for precise definitions of these periodic orbits, see Proposition \ref{thm_AC1D_PO}; when $c=0$, the latter two local MLPs degenerate to paths that cross separatrix at fixed points). Hyperbolicity of these fixed point and periodic orbits can be checked numerically.

\begin{figure}[h]
\centering
\footnotesize
\subfigure[null nucleation: separatrix crossing at a uniform saddle]{
\includegraphics[width=\textwidth]{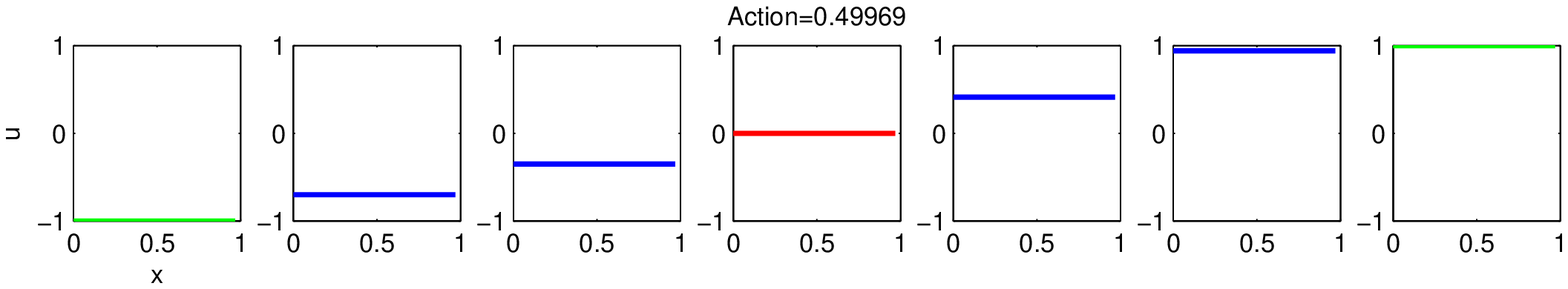}
}
\subfigure[single nucleation: separatrix crossing at periodic orbit \#1]{
\includegraphics[width=\textwidth]{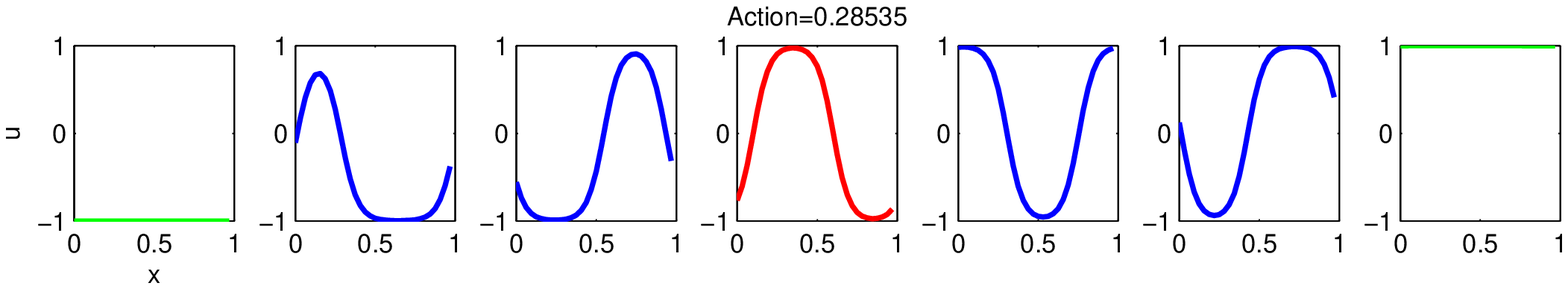}
}
\subfigure[double nucleation: separatrix crossing at periodic orbit \#2]{
\includegraphics[width=\textwidth]{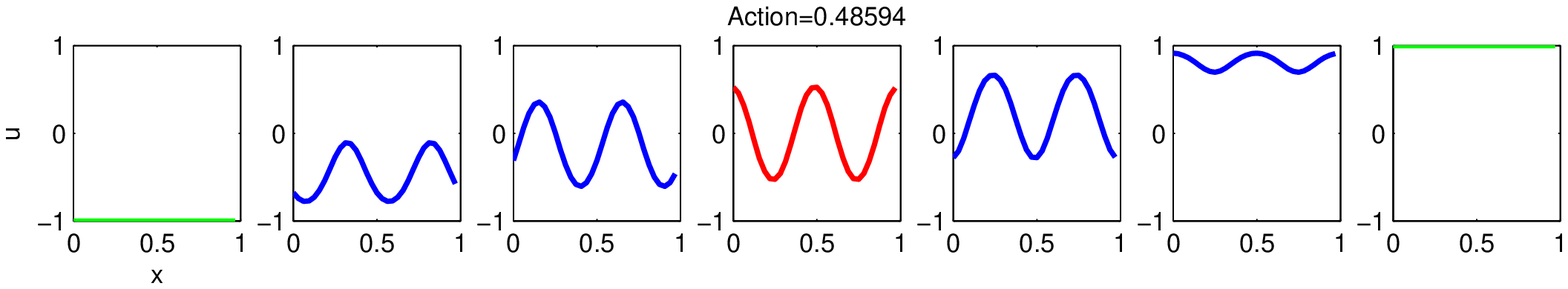}
}
\caption{\footnotesize Finite length approximations of MLPs between $u=-1$ and $u=1$ in SPDE \eqref{eq_AC1D}, computed by up-down gMAM (described in Section \ref{sec_results_gMAMextend}) with different intersections at the separatrix. $\kappa=0.005$, $c=0.1$. Each MLP is illustrated by seven snapshots, uniformly distributed from reparameterized time 0 to 1, with green indicating the stable points and red the intersection with the separatrix.}
% for 1bands, k=1/32, n1=100, n2=10, h=1e-3, dt=0.01
% for 2bands, k=1/32, n1=40, n2=10, h=0.01, dt=0.01 (Ex4_gMAMextend_kappa5e-2_k32_2bands)
% c.f. gMAM original 1band, k=1/32, n=100, h=1e-3 (Ex4_gMAM_kappa5e-2_k32_1bands)
\label{fig_AC1D_gMAM}
\end{figure}

The exact values of local action minima can be obtained by Theorem \ref{thm_action} once the separatrix crossing location is fixed. They are respectively 0.5, $\approx 0.2665$, and $\approx 0.4851$. Therefore, our numerically computed MLPs are rather accurate.

 %(`$\approx$' because in our case space is discretized pseudospectrally by 32 wavenumbers, and therefore points on the periodic orbits are approximations). Paths obtained by gMAM are approximations and their accuracies can be seen from action values. %Each path is a local minimizer of the action functional, whose corresponding transition rate depends where it intersects the separatrix. % One conjecture is, periodic orbit that bifurcated from a periodic solution of $\kappa u_{xx}+u-u^3=0$ that is closer to $u=0$ will have a larger action, and hence a smaller likelihood.

Note there has been a long time interest in studying metastable transition rate in this system. See for instance the classic paper of \cite{FaJo82}, where bounds of the action are analytically obtained for estimating the rate.

\subsection{2D-space advection-diffusion-reaction SPDE}
\label{sec_AC2D_intro}
Our final example generalizes the previous example to 2D-space. Numerical evidence suggests this nongradient system is no longer orthogonal-type (Section \ref{sec_AC2D}). The system is 
the following 2D-space 1D-time SPDE
\begin{equation}
    \phi_t=\kappa \Delta \phi+\phi-\phi^3+c \sin(2\pi y)\partial_x \phi +  \sqrt{\epsilon} \eta,
    \label{eq_AC2D}
\end{equation}
with periodic boundary conditions $\phi(x+1,y,t)=\phi(x,y,t)$ and $\phi(x,y+1,t)=\phi(x,y,t)$. $0<\kappa \ll 1$. Adding noise to nonlinear PDE with $\geq 2$ spatial dimension is nontrivial (see for instance \cite{walsh1986SPDE,da2008SPDE,hairer2012triviality}), and here we follow \cite{KoOt07}: $\eta:=\phi_\lambda *\eta'$ is a spatially regularized noise, where $*$ denotes convolution, $\phi_\lambda(x,y)=\lambda^{-2}\phi(x/\lambda,y/\lambda)$ with an approximate identity $\phi$, and $\eta'$ is spatiotemporal white-noise with covariance $\mathbb{E} [\eta(x,y,t)\eta(x',y',t')]=\delta(x-x')\delta(y-y')\delta(t-t')$. % After spatial discretization we do not have to worry about the regularization of noise any more.

This system is a 2D Allen-Cahn equation with additional shear and noise. Local MLPs through saddle points in this system have been studied in \cite{HeVa08c}. We now consider periodic orbits and whether they intersect with a MLP -- note Theorem \ref{thm_action} only applies to orthogonal-type systems and thus fails to answer this question.

We numerically identified hyperbolic periodic orbits in this system, which bifurcated from saddle points as $c$ increases. There are numerically identified local MLPs that intersect with these periodic orbits. However, detailed results are deferred to Section \ref{sec_AC2D}, as it is necessary to first introduce the employed numerical tools.

\section{Hyperbolic periodic orbit identification in general nongradient systems by p-String method}
\label{sec_pString}
This section modifies String method to identify hyperbolic periodic orbits. The method applies to general nongradient systems not restricted to orthogonal-type, as long as the orbit of interest is attracting on the separatrix between two attraction basins.

% not only facilitates the understanding of metastable transitions, but also is 
The identification of periodic orbits is an important problem on its own, because they are the second simplest class of limit sets and characterize a dynamical system's behaviors. Stable periodic orbits can often be obtained as limits of numerically integrated initial value problems, and similarly fully unstable periodic orbits can be found by integration backward in time. Hyperbolic periodic orbits, however, are more difficult to compute. One popular approach is to solve a boundary value problem, oftentimes via a combination of shooting method and optimization techniques (e.g., \cite{lust1998adaptive,davidchack1999efficient,guckenheimer2000computing,ambrose2010computation}). Methods based on parameterization and Fourier series have also been used (e.g., \cite{haro2016parameterization,haro2007parameterization} and \cite{viswanath2001lindstedt}). In addition, there are approaches based on geometric / topological considerations (e.g., \cite{vrahatis1995efficient}). However, as the unstable manifold of the periodic orbit increases in dimension (e.g., Sections \ref{sec_AC1D} and \ref{sec_AC2D_intro}), performances of these methods oftentimes deteriorate.

We adopt an alternative approach based on the augmented dynamical system that String method constructs. In this system, a hyperbolic periodic orbit of the original system becomes part of a stable limit set, and numerics are thus enabled.

\subsection{The method}
\label{sec_POid}
Consider
\begin{equation}
    \dot{x}=f(x)
    \label{eq_generalDynamicalSystem}
\end{equation}
with smooth enough $f(\cdot)$. Suppose this dynamical system contains two stable fixed points $x_a$ and $x_b$, and their basins of attractions $\mathcal{D}_a$ and $\mathcal{D}_b$ cover the entire phase space (i.e., $\overline{\mathcal{D}_a} \bigcup \overline{\mathcal{D}_b} = \mathbb{R}^d$). Denote the separatrix submanifold by $\mathcal{S}$, i.e., $\mathcal{S}=\partial \mathcal{D}_a \bigcap \partial \mathcal{D}_b$. Assume there is at least one fixed point $x_s$ or periodic orbit $x_{PO}(t)$ that is attracting on $\mathcal{S}$.

$\omega$-limit sets of the dynamics restricted on the separatrix submanifold, such as fixed point or periodic orbit, are unstable in the full phase space. Due to this instability, numerical errors make it difficult to locate these limit sets. We use the following algorithm to approximate such hyperbolic fixed point or periodic orbit:

\paragraph{p-String method:}
\begin{enumerate}
\item
    Evolve a discretized path from $x_a$ and $x_b$ (i.e., the initial path $\phi_0(\cdot)$ satisfies $\phi_0(1)=x_a$ and $\phi_0(n+1)=x_b$, with $n$ sufficiently large) by String method (see e.g., \cite{StringMethod,StringSimplified} or Appendix \ref{sec_review_String}) --- that is, alternate between two substeps: evolution of each point on the path by \eqref{eq_generalDynamicalSystem}, and reparametrization of the path.\label{item_step1}
\item
    Terminate the evolution when convergence towards a periodic evolution is detected (including the degenerate case of converged String evolution).

    Specifically, at each step $i$, compute the action $S_i$ of the current path $\phi_i$. If there is $\tilde{i}<i$ such that $\big| S_i-S_{\tilde{i}}\big| /\max\{|S_{\tilde{i}}|,|S_i|\}<\text{threshold}$, then trigger an additional check %\footnote{The purpose of this additional check is to avoid false alarm. Although its computational cost is negligible, for high dimensional problems the cost of storing the evolution history could be big. In experiments with all examples in this article, no false alarm was encountered, and this additional check is unnecessary.}
    on whether $\| \phi_i(\cdot) - \phi_{\tilde{i}}(\cdot) \|$ is small enough; if yes, then periodic behavior is detected and String evolution terminates.
    \label{item_step2}
\item
    Denote by $f$ the step at which String evolution was terminated. Store the last path $\phi_f(\cdot)$. Further evolve each point on this path, i.e., $\phi_f(j)$ for $1\leq j \leq n+1$, according to the dynamics of \eqref{eq_generalDynamicalSystem}, however this time without reparametrization.
    \label{item_step3}
\item
    Terminate the evolution when all points but one are attracted to $x_a$ or $x_b$. That is, at each step $i$, compute $d_j=\min\{ \|\phi_i(j)-x_a\|, \|\phi_i(j)-x_b\| \}$ for each $1 \leq j \leq n+1$. Terminate when $\{j \big| |d_j|> h \}$ contains only one element $j^*$, where $h$ is the evolution timestep.
    \label{item_step4}
\item
    Output $\phi_f(j^*)$ as the result. It is a fixed point if $\tilde{f}=f-1$. Otherwise it is a point on a periodic orbit, whose period is approximately $(f-\tilde{f}-1)h$; this periodic orbit can be recovered by evolving $\phi_f(j^*)$ according to \eqref{eq_generalDynamicalSystem}.
    \label{item_step5}
\end{enumerate}

\begin{Remark}
    The method still works if $x_a$ and $x_b$ are not exactly the two sinks but in different attraction basins. In addition, if the attractors of the two basins are not points, e.g., limit cycles instead, the method can still work when $n$ is large enough and steps \ref{item_step3}-\ref{item_step5} are modified accordingly; however, an $n$ too large might lead to inefficient computations.
\end{Remark}

\begin{Remark}
	The accuracy of p-String method increases with $n$. However, larger $n$ corresponds to more computations. Two possible improvements are, (i) an adaptive version of p-String, in which points on the path away from the separatrix are discarded, so that more points can be placed near the separatrix, and (ii) to use the p-String result of as the initial condition of some other high-fidelity method (e.g., Newton or quasi-Newton based; thanks to an anonymous referee's comment).
\end{Remark}

\subsection{The rationale}
\label{sec_pString_rationale}
The algorithm contains two parts. Steps \ref{item_step1}-\ref{item_step2} are based on the idea that the limit set of String evolution dynamics contains a limit set of \eqref{eq_generalDynamicalSystem}. More specifically, consider two evolutions of paths, one without reparametrization and one with:
\begin{eqnarray*}
    &\psi_t(\alpha,t)=f(\psi(\alpha,t)), & \qquad \psi(0,t)=x_a, \psi(1,t)=x_b \\
    &\phi_t(\beta,t)=f(\phi(\beta,t))+r(\beta,t), & \qquad \psi(0,t)=x_a, \psi(1,t)=x_b,
\end{eqnarray*}
where $r(\beta,t)$ is a virtual force parallel to $\phi_\beta$ for ensuring a constant distance parametrization $\|\phi_\beta\|=\text{constant}$. Geometrically, $\psi(\cdot,t)$ and $\phi(\cdot,t)$ represent the same path in phase space, i.e., for any $\alpha\in[0,1]$, there exists a $\beta(t)\in[0,1]$ such that $\phi(\beta(t),t)=\psi(\alpha,t)$. It is just their parametrizations that are different: for large $t$, $\phi(\cdot,t)$ is much less singular than $\psi(\cdot,t)$. String method computes $\phi$ due to numerical considerations \cite{StringMethod,StringSimplified}.

Note $\psi(\cdot,t)$ has to cross the separatrix, supposedly at $\alpha_0(t)$. $\alpha_0(t)$ is in fact a constant $\alpha_0$, because separatrix is invariant under dynamics. Given $T$ large enough, $\psi(\alpha_0,T)$ will approach an $\omega$-limit set on the separatrix. Since there is $\beta_0(T)$ such that $\phi(\beta_0(T),T)=\psi(\alpha_0,T)$, $\phi(\beta_0(T),T)$, which is a point on the path given by String evolution, approximates a point on the limit set.

The second part of the algorithm (Steps \ref{item_step3}-\ref{item_step5}) finds $\beta_0(T)$ and thus $\phi(\beta_0(T),T)$. The idea is, if a discretized path is evolved pointwise under \eqref{eq_generalDynamicalSystem} without reparameterization, points away from the separatrix will soon be attracted towards $x_a$ or $x_b$, and the point that remains most further away from $x_a$ and $x_b$ corresponds to what's on the separatrix. Numerical error of this identification naturally decreases as path discretization is refined.

\subsection{Example results}
\label{sec_pString_res}
% CSM does not work for 3D2 SDE examples. tracking $\alpha$ does not work either.

\paragraph{2D SDE system.} p-String method approximates the $(0,1)$ saddle as
\[
    z=-0.0000\ldots, r=1.0000\ldots
\]
Computation used $n=30$ discretization points, $h=0.01$ step size forward Euler evolution, $\text{threshold}=10^{-6}$, and initial path linear from $(-1,0)$ to $(1,0)$.

\paragraph{3D SDE system (rotational).}
p-String method identifies a point on the periodic orbit as
\[
    x=-0.7115\ldots, y=0.7018\ldots, z=0.0000\ldots
\]
It corresponds to $r=\sqrt{x^2+y^2}=0.9993\ldots$. Recall the true periodic orbit is $r=1,z=0$.

Computation used $n=50$, $h=0.01$ Verlet evolution, $\text{threshold}=10^{-6}$, and random initial path from $(0,0,-1)$ to $(0,0,1)$ (random for avoiding singularity at $|z|\neq 1,x=y=0$).

\paragraph{3D SDE system (non-rotational).}
p-String method identifies a point on the periodic orbit as
\[
    x=1.9602\ldots, y=-1.0365\ldots, z=-0.0038\ldots
\]
It corresponds to $(x^4+y^4)^{1/4}=1.9974\ldots$. Recall the true periodic orbit is $(x^4+y^4)^{1/4}=2,z=0$.

\begin{figure}[h]
    \centering
    \includegraphics[width=\textwidth]{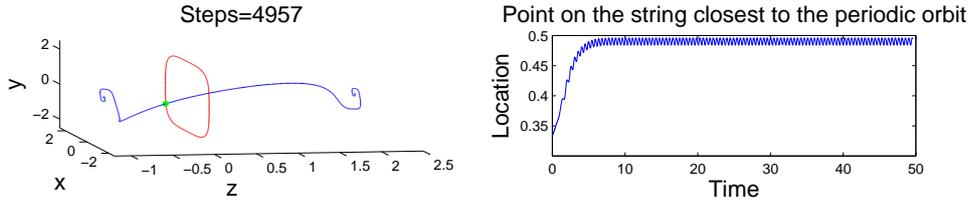}
    \caption{\footnotesize Identification of periodic orbit in system \eqref{eq_system3D2} by p-String method. Left figure illustrates the path at termination of the string evolution. Right figure illustrates where intersection between the path and the separatrix is located on the path.}
    \label{fig_Ex3_3D2_string}
\end{figure}

Figure \ref{fig_Ex3_3D2_string} left panel illustrates the terminal configuration of the string $\phi_f$. Three facts are: (i) it is not MLP; (ii) it is not necessarily perpendicular to the separatrix, even though $-\nabla V$ is perpendicular to the separatrix; (iii) the separatrix crossing location on the string (i.e. $\beta_0(T)$ in Section \ref{sec_pString_rationale}) is not a constant; instead, since the string evolution converges to a limit cycle, $\beta_0(\cdot)$ converges to an oscillation.

Computation was done using $n=1000$, $h=0.01$ Verlet evolution, $\text{threshold}=10^{-6}$, and initial path linear from $(0,0,-1)$ to $(0,0,2)$. We chose a large $n$ only so that $\beta_0(T)$ has three significant digits and Figure \ref{fig_Ex3_3D2_string} right panel is numerically smooth.

\paragraph{1D-space SPDE.}
According to Proposition \ref{thm_AC1D_PO}, any point $u(x)$ on the true periodic orbit satisfies $\kappa u_{xx}+u-u^3=0$. With linear/vertical/double vertical initial string configurations (see Appendix \ref{sec_initialPath}), p-String method identified $u$ that numerically satisfy this equation ($L^2$ residuals: $0$ / $0.0033\ldots$ / $0.0014\ldots$) and govern the MLPs of null/single/double nucleation. See red in Figure \ref{fig_AC1D_gMAM} for graphs of these $u$'s.

Simulation settings are: $n=40$, $h=0.01$, $\text{threshold}=10^{-6}$, space is pseudospectrally discretized using first 32 modes, and time integration is based on Strang splitting, where exponential integrator is used for diffusion and two half-step Eulers are used for reaction and advection.

\paragraph{2D-space SPDE.} See Section \ref{sec_AC2D}.

\section{Identified periodic orbit helps understand metastable transitions}
\label{sec_results}

\subsection{Transition rate}
The transition rate from $x_a$ to $x_b$ is quantified by the quasipotential up to a prefactor. If one ignores the prefactor, which is generally not provided by a large deviation theory (e.g., \cite{DeZe98}), then it is sufficient to investigate the minimum action.

For systems of orthogonal-type (including gradient systems), once a fixed point or periodic orbit that attracts on the separatrix is identified, there is an associated local minimum of action, expressed in terms of a barrier height (Theorem \ref{thm_action}; this is analogous to Arrhenius rate formula \cite{hanggi1990reaction}). Therefore, if one only cares about the transition rate, computation of the corresponding transition path is not necessary. However, to obtain the global minimum of action (and hence the exponent in the transition rate), one has to exhaust attractors on the separatrix, which could be challenging for high dimensional problems.

For general non-orthogonal-type systems, however, there may be multiple local minima of the action associated with one separatrix crossing location $x_s$. Section \ref{sec_AC2D} contains an example. Clearly, no single function of $x_a$ and $x_s$ can provide such multiple local minimum values. It's unclear whether the global minimum (i.e., the quasipotential) is some barrier height.

\subsection{Transition path and its numerical computation}
\label{sec_results_gMAMextend}
The MLP from $x_a$ to $x_b$ bridges two attraction basins and thus has to cross the separatrix. If the crossing location is known (denoted by $x_s)$, the MLP can be numerically obtained more efficiently. This is because MLP can be made parameterization-independent (see \cite{HeVa08a, VaHe08b, HeVa08c} or Section \ref{sec_orthogonal_MLP}), and the concatenation of two MLPs, first from $x_a$ to $x_s$ and then from $x_s$ to $x_b$, corresponds to the MLP from $x_a$ to $x_b$. Finding these two shorter MLPs generally requires less exploration in the state space.

$x_s$ is thus helpful. For orthogonal-type systems, Theorem \ref{thm_action} implies that a fixed point or a point on a periodic orbit that attracts on the separatrix is the $x_s$ of a local MLP. Such a point can be identified by p-String method (Section \ref{sec_pString}). For general systems, given any local MLP, there is always a path with the same action that crosses a point in the $\omega-$limit set (i.e., attractor) on the separatrix, because one can modify the given local MLP by adding a link from its separatrix crossing location to the $\omega-$limit set of that crossing location, without spending any additional noise. However, it is left unproved whether any point in the attractor corresponds to a local MLP. In our numerical experiments, various hyperbolic $x_s$ (attracting on the separatrix) do appear to correspond to local MLPs (see Section \ref{sec_AC2D}).

We use gMAM to obtain the $x_a \rightarrow x_s$ MLP. To obtain the $x_s \rightarrow x_b$ MLP, note it is in fact a zero of the action functional, no matter whether the system is of orthogonal-type. This is because no noise is required once the system is in the attraction basin of $x_b$. Therefore, theoretically speaking, one only needs to find the $x_a$ to $x_s$ MLP, and then compute the stable heteroclinic orbit from $x_s$ to $x_b$. Numerically, a perturbation of $x_s$ in the attraction basin of $x_b$ is needed, so that downhill heteroclinic orbit can be computed in finite time. There are multiple ways to find such a perturbation, and we choose to use a coarse gMAM computation. More precisely, we consider $dx=f(x)dt+\sqrt{\epsilon}dW$ and use the following algorithm:

% Note: in the two 3D examples, gMAM does not converge to a fixed path. Instead, it converges to a limit cycle of paths, whose intersections with the separatrix approximate $r=1$ (i.e., the periodic orbit). This is because the action has a continuum of minimizers due to rotational symmetry. At each step, gMAM tries to converge to one of these minimizers, but then due to reparametrization, it at the next step tries to converge to a slightly different minimizer. One way to make gMAM terminate is to employ a threshold on change in action values, rather than in paths. Another way is to fix the crossing point on the periodic orbit -- this is our proposed approach, which will be described in Section \ref{sec_results}.

\paragraph{Up-down gMAM for computing MLP that crosses separatrix at $x_s$:}
\begin{enumerate}
\item
    Choose a priori three positive integer parameters: $\Delta$, $n_1$ and $n_2$; in general, $n_1 \gg n_2$.
\item
    Compute a MLP from $x_a$ to $x_s$ using gMAM with paths discretized by $n_1+1$ points. Denote by $x^{\text{up}}_j$ ($1\leq j \leq n_1+1$) the resulting path.
    \label{item2_step2}
\item
    Compute a MLP from $x_s$ to $x_b$ using gMAM with paths discretized by $n_2+1$ points. Denote by $\tilde{x}^{\text{down}}_j$ ($1\leq j \leq n_2+1$) the resulting path.
    \label{item2_step3}
\item
    Let $x^+=\tilde{x}^{\text{down}}_2$. Integrate $\dot{x}=f(x)$ with initial condition $x(0)=x^+$. The integration step size $\delta t$ is usually chosen the same as the one used for gMAM evolutions. Terminate the integration at the smallest time $T$ satisfying $\|x(T)-x_b\| \leq \delta t$. Denote by $\hat{x}_k$ ($0 \leq k \leq T/\delta t$) the numerically integrated discrete trajectory.
    \label{item2_step4}
\item
    Let
    \[
        x^{\text{down}}_j=\begin{cases}
            \hat{x}_{(j-1)\Delta}, & \qquad 1\leq j \leq n_2; \\
            x_b, & \qquad j=n_2+1
        \end{cases},
    \]
    where $n_2=T/\delta t/\Delta+1$. That is, down sample once every $\Delta$ points in $\hat{x}$ to form $x^{\text{down}}$. If there is no requirement on the number of discretization points, $\Delta$ can simply be chosen as $1$.
\item
    Form a discrete path $x_j$ by
    \[
        x_j=\begin{cases}
            x^{\text{up}}_j, & \qquad 1\leq j \leq n_1+1; \\
            x^{\text{down}}_{j-n_1-1}, & \qquad n_1+2 \leq j \leq n+1
        \end{cases},
    \]
    where $n=n_1+n_2+1$.
\item
    Compute a numerical approximation of the geometrized action
    \[
        \mathcal{S}=\int_0^1 \| \dot{x} (\alpha)\| \|f(x(\alpha))\| - \langle \dot{x} (\alpha), f(x(\alpha)) \rangle d\alpha
    \]
    using quadrature and finite difference for $\dot{x}$. Be mindful that $\alpha\in [0,1]$ is discretized to $j$ using non-uniform grid sizes, $\frac{1}{2n_1}$ in $x^{\text{up}}$ region, and $\frac{1}{2n_2}$ in $x^{\text{down}}$ regions.
\end{enumerate}

% String cannot obtain even downhill due to infinite length (loops more and more around PO -- not informative)

\paragraph{Comparison with the original gMAM:}
Table \ref{tab_gMAMcomparsion} compares gMAM and up-down gMAM (with its gMAM component implemented in the same way; $x_s$ is from Section \ref{sec_pString_res}, computed by p-String method).

\begin{table}[h]
\hspace{-20pt}
\small
\begin{tabular}{l| c | c | c | c | c}
System              & 2D SDE    & 3D SDE            & 3D SDE                & 3D SDE            & 1D-space SPDE \\
                    &           & rotational        & non-rotational        & non-rot. (finer)  & $\kappa=0.01$,$c=0.1$ \\
\hline
$n_1$               & 100       & 100               & 100                   & 200               & 40 \\
Total evolution steps$_1$           & 79        & 36675             & 21074                 & 74130             & 37968 \\
$n_2$               & 10        & 20                & 20                    & 40                & 10 \\
Total evolution steps$_2$           & 78        & 4679              & 4099                  & 5646              & 6676 \\
$\Delta$            & 1         & 1                 & 1                     & 1                 & 10 \\
up-down gMAM action & 0.50008   & 0.50031           & 0.85490               & 0.84123           & 0.37873 \\
\hline
$n$                 & 100       & 100               & 100                   & 200               & 40 \\
Total evolution steps & 148       & 38661$^*$         & 43140$^*$             & 166123$^*$        & 16796 \\
gMAM action         & 0.49987   & 0.50448           & 0.86163               & 0.84545           & 0.39827 \\
\hline
Termination threshold & $10^{-6}$ & $10^{-5}$       & $10^{-6}$             & $10^{-6}$         & $10^{-6}$ \\
Evolution step size & $0.1$     & $0.01$            & $0.01$                & $0.01$ and $0.005 ^\dagger$ & $0.01$ \\
\hline
True action         & 0.5       & 0.5               & $5/6 \approx 0.8333$  & $5/6 \approx 0.8333$ & $\approx 0.3732$ \\
\end{tabular}
\caption{Comparison between gMAM and up-down gMAM.}
*: gMAM path evolution was terminated when the action values converge (threshold on the amount of change: $10^{-10}$ for 3D SDE (rot.) and $10^{-8}$ for 3D SDE (non-rot.)) \\
% However, in our experiments, gMAM does not converge to a fixed path. Instead, it converges to a limit cycle of paths, whose intersections with the separatrix approximate $r=1$ (i.e., the periodic orbit). This is because the action has a continuum of minimizers due to rotational symmetry. At each step, gMAM tries to converge to one of these minimizers, but then due to reparametrization, it at the next step tries to converge to a slightly different minimizer. One way to make gMAM terminate is to employ a threshold on change in action values, rather than in paths. Another way is to fix the crossing point on the periodic orbit -- this is our proposed approach, which will be described in Section \ref{sec_results}.
$\dagger$: up-down gMAM uses $h=0.01$ and gMAM uses $h=0.005$, because gMAM with $0.01$ is no longer convergent -- there action value oscillates around 0.89.
\label{tab_gMAMcomparsion}
\end{table}

\begin{figure}[h!]
\centering
\footnotesize
\subfigure[n=50]{
\includegraphics[width=0.48\textwidth]{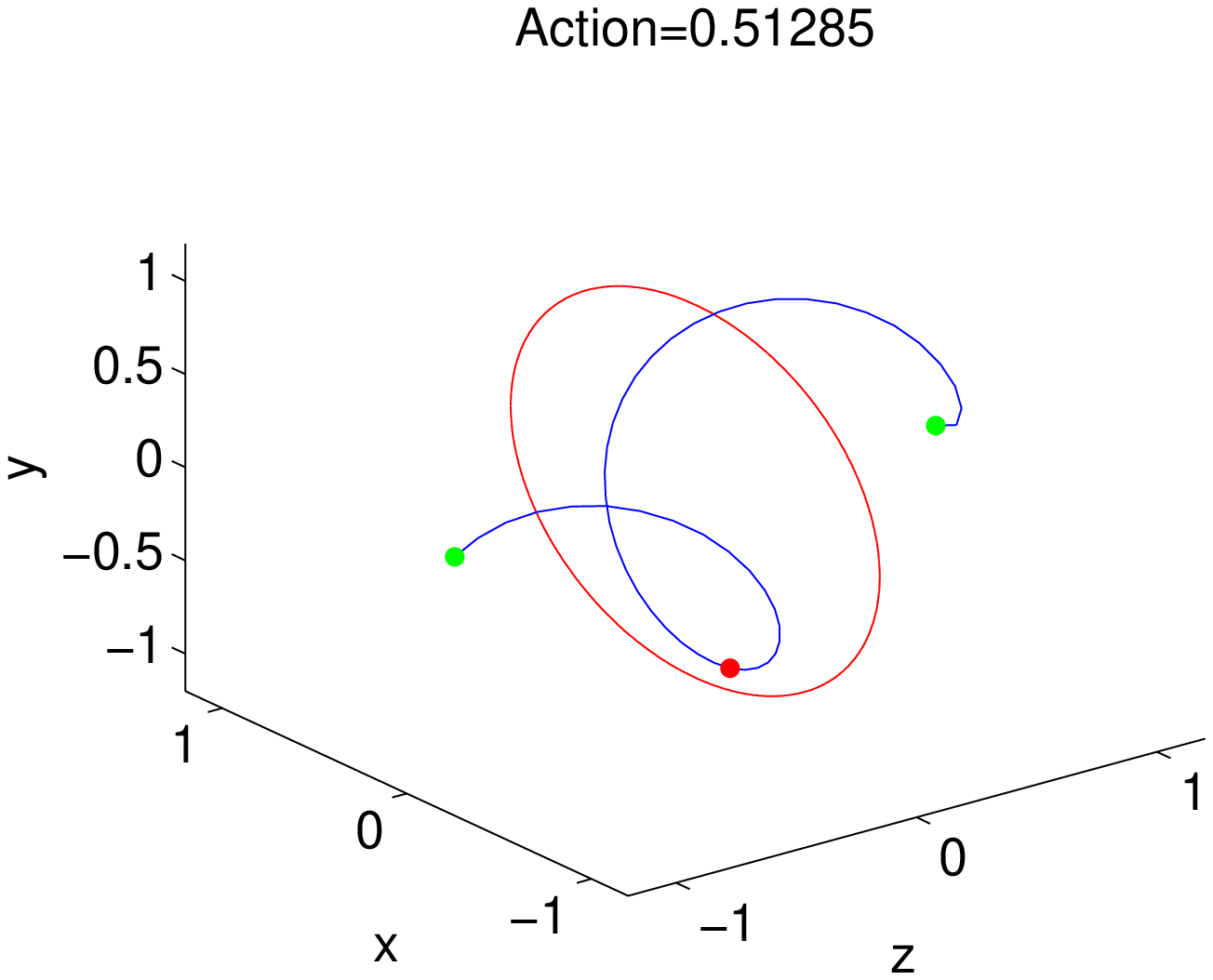}
}
\subfigure[n=100]{
\includegraphics[width=0.48\textwidth]{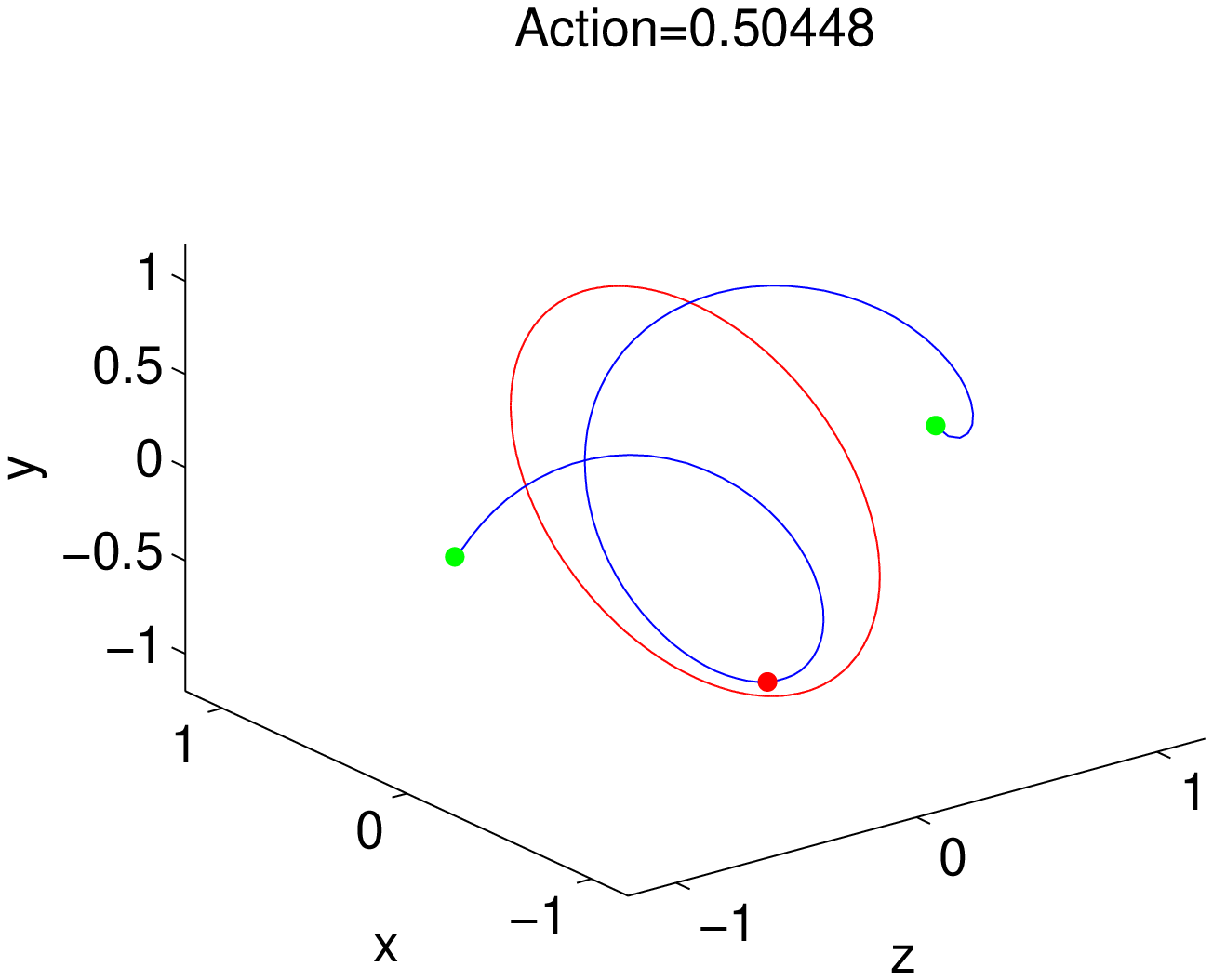}
}
\caption{\footnotesize MLPs in 3D SDE system \eqref{eq_system3D} approximated by gMAM.}
\label{fig_3DgMAMoriginal}
% h=0.01
\end{figure}

Up-down gMAM demonstrates better accuracy in minimizing the action. The reason is, gMAM-approximated MLP intersects the separatrix at a location further from the periodic orbit (compare Figure \ref{fig_3DgMAMoriginal} with the $n=100$ up-down gMAM results in Figure \ref{fig_3DgMAM}). This intersection can be made closer to the periodic orbit by increasing $n$ in gMAM. However, p-String is much more accurate (the separatrix crossing location used by up-down gMAM in Figure \ref{fig_3DgMAM} was computed by p-String with $n=50$). p-String suits the identification of separatrix crossing better, because a local MLP through a periodic orbit is of infinite arclength, and gMAM has to compromise and approximate it by a path of finite arclength; on the contrary, p-String not necessarily approximates the MLP and the string can be of finite length (this is in fact provable for the example in Section \ref{sec_examples_3D}, and the long-time string evolution will just be Figure \ref{fig_2Dstring} rotating in the 3D space).

Up-down gMAM also appears to be more efficient. It usually converges faster, and sometimes allows larger time step too. When same step size is used, the total computational cost of up-down gMAM can be characterized by $n_1 \cdot \text{steps}_1+n_2 \cdot \text{steps}_2$, because the cost of up-down gMAM Step \ref{item2_step4} is negligible comparing to Step \ref{item2_step2} and \ref{item2_step3}, and gMAM cost is characterized by $n \cdot \text{steps}$. In most of our experiments, $n_1 \cdot \text{steps}_1+n_2 \cdot \text{steps}_2 < n \cdot \text{steps}$. To understand this comparison in a fair way, however, note (i) gMAM works for more general problems, e.g., finding MLP between arbitrary points; (ii) up-down gMAM uses additional information on $x_s$, whose computation by p-String method also takes time --- this cost, however, was much smaller in our experiments than that of up-down gMAM Step \ref{item2_step2} and \ref{item2_step3}. To provide an illustration, for generating results in the ``3D SDE non-rotational'' column in Table \ref{tab_gMAMcomparsion} on an Intel i7-4600 laptop with MATLAB R2016b, gMAM took 108.6 seconds, while p-String took 16.7 seconds and up-down gMAM took 54.8 seconds (altogether: 71.5 seconds, 65.8\% of gMAM); for the ``3D SDE non-rot. (finer)'' column, gMAM took 847.4 seconds, while p-String took 16.7 seconds and up-down gMAM took 378.0 seconds (altogether: 394.7 seconds, 46.6\% of gMAM).

\subsection{MLPs in a non-orthogonal-type system \eqref{eq_AC2D}}
\label{sec_AC2D}
Although local MLPs in orthogonal-type systems can be understood by Theorem \ref{thm_action}, in general nongradient systems, it is unclear whether an arbitrary point in the attractor on the separatrix still corresponds to a local MLP. We numerically demonstrate that system \eqref{eq_AC2D} is not of orthogonal-type, and yet there are still local MLPs that cross identified saddle points and periodic orbits.

\paragraph{Fixed points and periodic orbits.}
Without noise, system \eqref{eq_AC2D} is
\begin{equation}
    \phi_t=\kappa (\phi_{xx}+\phi_{yy})+\phi-\phi^3+c \sin(2\pi y)\partial_x \phi .
    \label{eq_AC2D_deterministic}
\end{equation}
Viewed as a dynamical system in $t$, $u(x,y)=-1$, $u(x,y)=1$ and $u(x,y)=0$ are fixed points. When shear is absent (i.e., $c=0$), we know $u=\pm 1$ are sinks, and $u=0$ is a saddle.

\begin{figure}[h]
    \centering
    \includegraphics[width=0.25\textwidth]{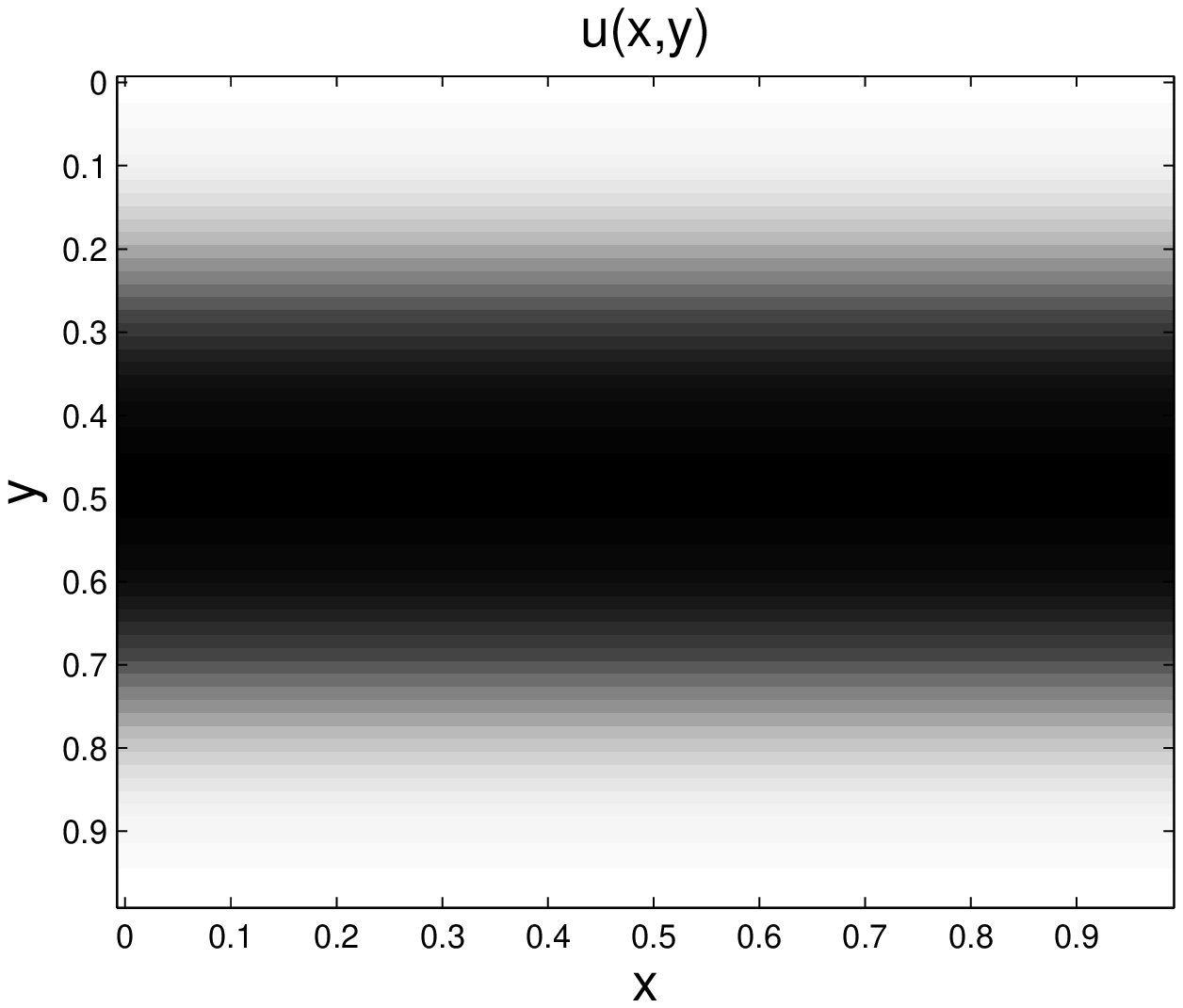} ~
    \includegraphics[width=0.25\textwidth]{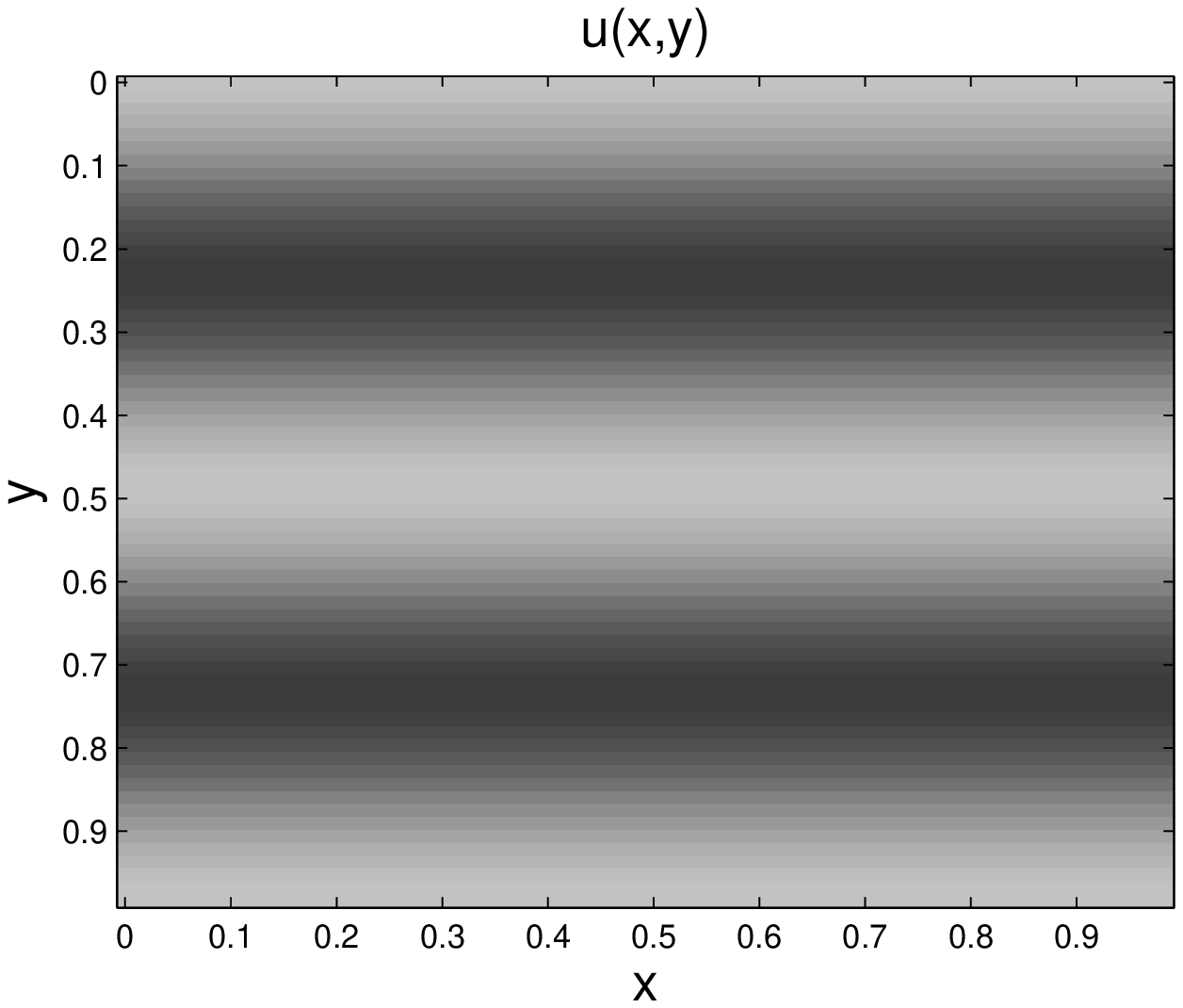}
    \caption{\footnotesize Horizontal fixed points numerically obtained by p-String method. $\kappa=0.005$. $c$ value is irrelevant. Values are represented by gray scale, white $-1$ and black $+1$.}
    % n=40, h=1e-2, k=1/64, kappa=0.005, dynamics termination: only 1 step
    \label{fig_CP_horizontal}
\end{figure}

There are also non-uniform fixed points. One group of them are invariant in $x$ and independent of $c$ values (see Figure \ref{fig_CP_horizontal}). It is straightforward to obtain them:

\begin{Proposition}[Horizontal fixed points]
    Any fixed point $\varphi(y,t)\equiv v(y)$ in 1D-space subsystem $\varphi_t=\kappa \varphi_{yy}+\varphi-\varphi^3$ (quantified in Proposition \ref{thm_AC1D_fixedPt2}) corresponds to a fixed point $\phi(x,y,t) \equiv u(x,y) :=v(y)$ in \eqref{eq_AC2D_deterministic}.
\end{Proposition}

\begin{figure}[h]
    \centering
    \includegraphics[width=0.25\textwidth]{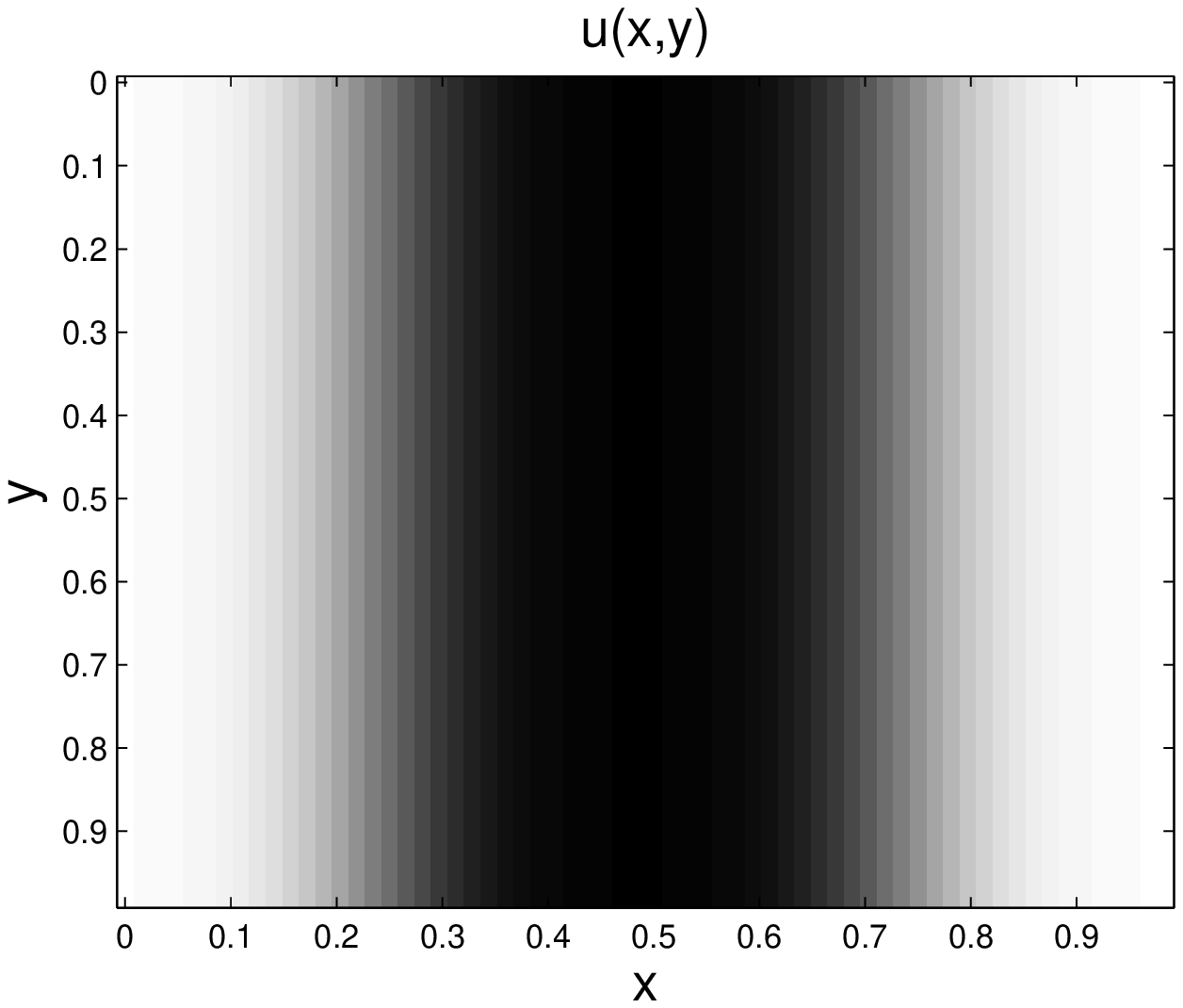} ~
    \includegraphics[width=0.25\textwidth]{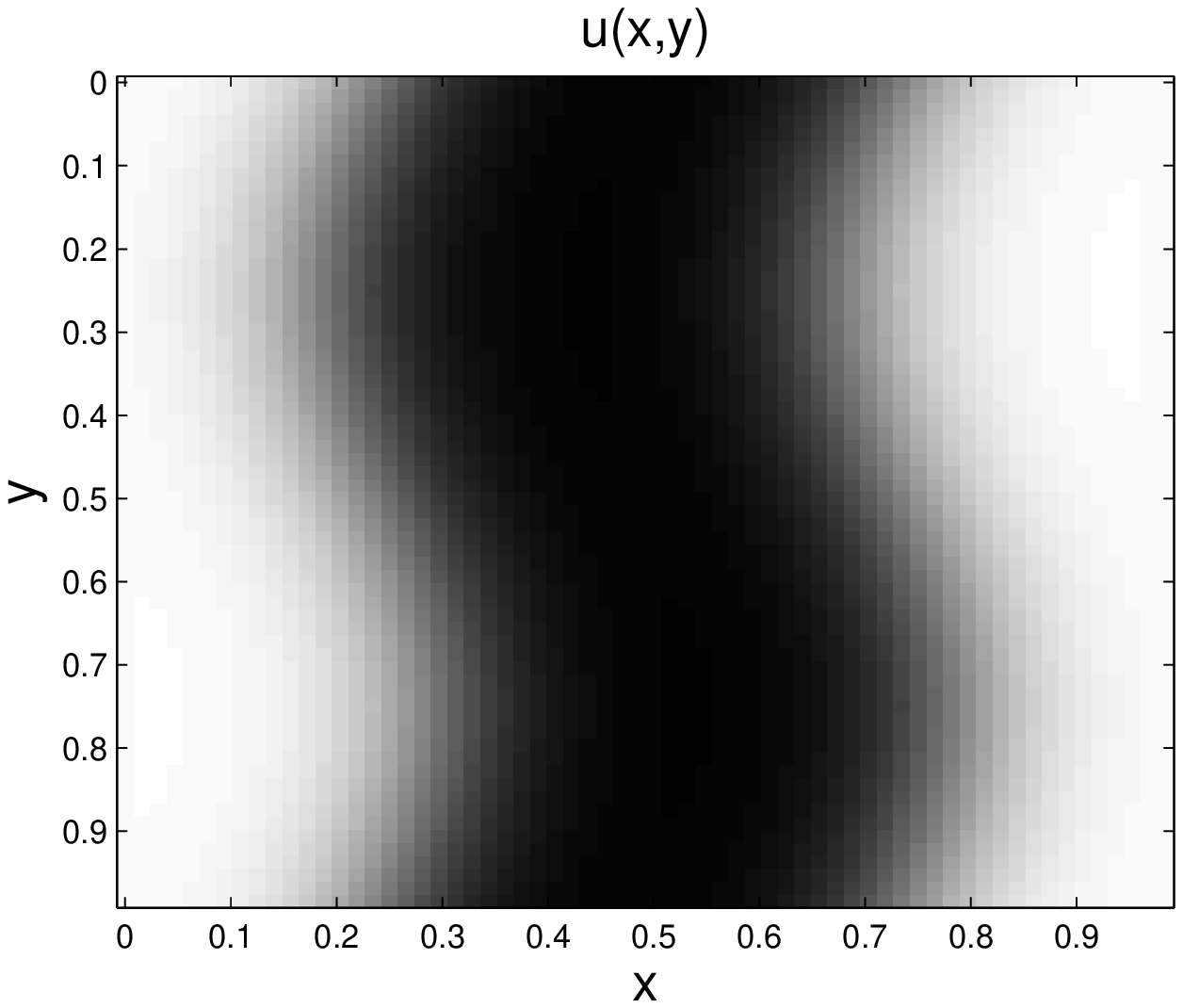} ~
    \includegraphics[width=0.25\textwidth]{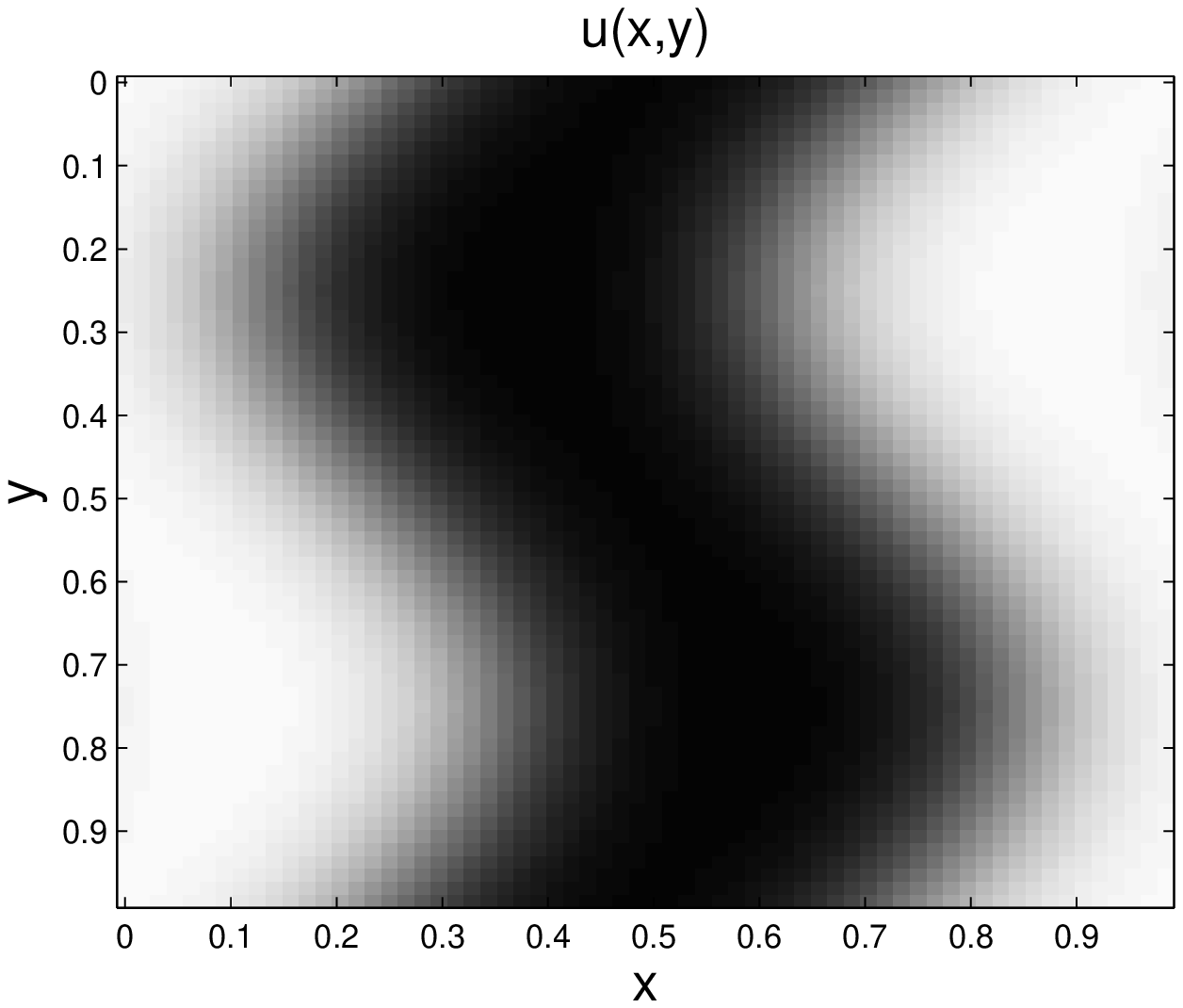} \\
    \includegraphics[width=0.25\textwidth]{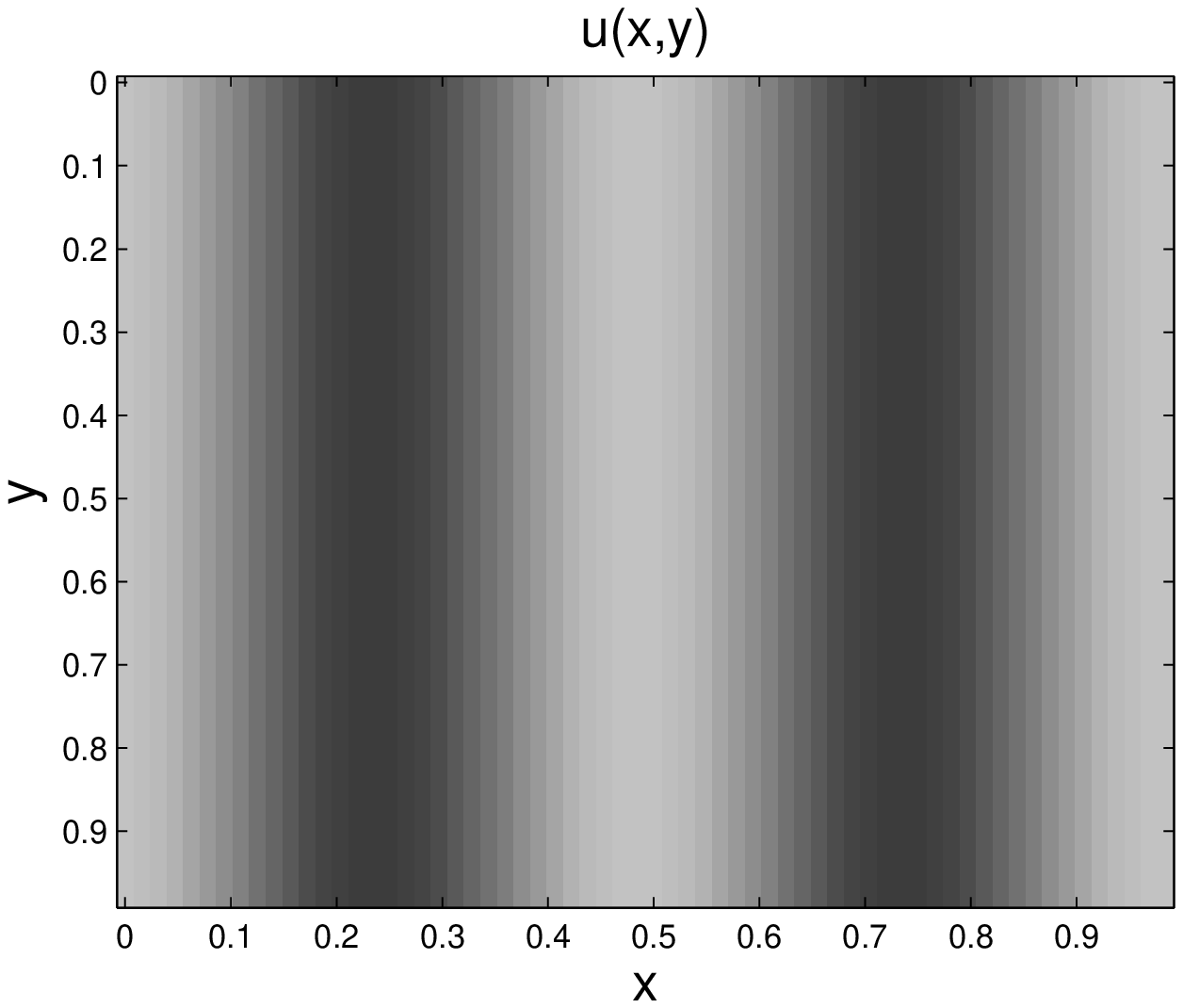} ~
    \includegraphics[width=0.25\textwidth]{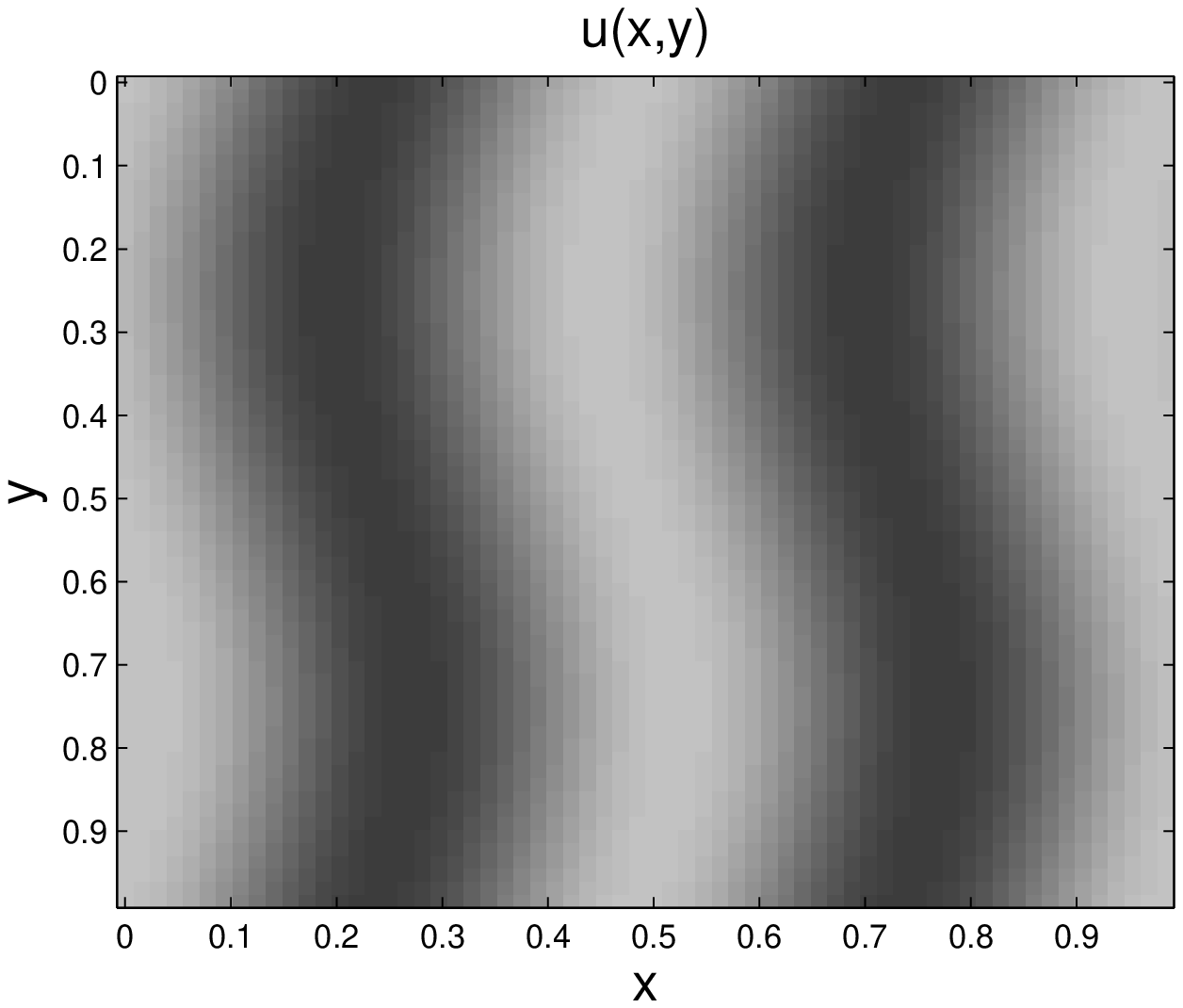} ~
    \includegraphics[width=0.25\textwidth]{blank}
    \caption{\footnotesize Vertical fixed points numerically obtained by p-String method. $\kappa=0.005$. In 1st column, $c=0$; in 2nd column, $c=0.01$; in 3rd column, $c=0.02$, and there is no longer a fixed point with 2-nucleations.}
    % n=40, h=1e-2, k=1/64, kappa=0.005, dynamics termination: only 1 step
    \label{fig_CP_vertical}
\end{figure}

A second group are almost invariant in $y$, sheared to an extent determined by small $c$. See Figure \ref{fig_CP_vertical} for an illustration, where these fixed points are numerically identified by the p-String method (Section \ref{sec_POid}) with vertical initial path (see Appendix \ref{sec_initialPath}). Their existences at small $c$ values are suggested by the following proposition (provable by a simple Taylor expansion):
\begin{Proposition}[Vertical fixed points; asymptotic]
    Let $v(\cdot)$ be a 1-periodic solution to $\kappa v_{xx}+v-v^3=0$. When $c$ is small enough,
    \[
        \tilde{u}(x,y)=v\left(x+\frac{c}{4\pi^2 \kappa}\sin(2\pi y)\right)+o(c)
    \]
    satisfies $\kappa (\tilde{u}_{xx}+\tilde{u}_{yy}) +\tilde{u}-\tilde{u}^3+c \sin(2\pi y)\partial_x \tilde{u}=0$.
    \label{thm_verticalFixPtAsymptotics}
\end{Proposition}

\begin{figure}[h]
    \centering
    \includegraphics[width=\textwidth]{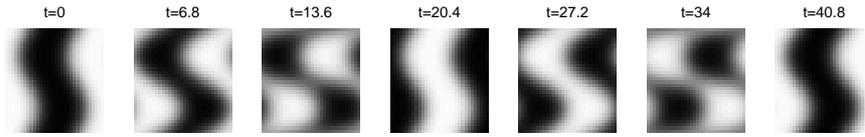}
    \caption{\footnotesize Snapshots of a periodic orbit that bifurcated from 1-nucleation vertical fixed point. $\kappa=0.005$, $c=0.05$.}
    % n=100, h=1e-2, k=1/32, TolErr=1e-8, dynamics termination: only 1 step
    \label{fig_PO_vertical}
\end{figure}

\begin{figure}[h!]
    \centering
    \includegraphics[width=0.5\textwidth]{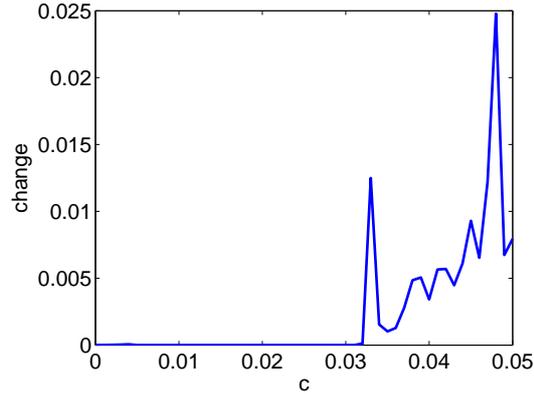}
    \caption{\footnotesize Bifurcation of fixed point to periodic orbit at $c\approx 0.032$, illustrated by $\|\phi_h-\phi\|_2$ as a function of $c$, where $\phi$ is the result of converged p-String method (with $10^{-8}$ tolerance threshold), and $\phi_h$ is a $h=0.1$ step evolution of $\phi$. $\kappa=0.005$, $c$ value is sampled from 0 to 0.05 with 0.001 increment.}
    % n=40, h=1e-2, k=1/16, TolErr=1e-8 (default is 1e-6), dynamics termination: only 1 step
    \label{fig_PO_cBifurcation}
\end{figure}

Periodic orbits are also numerically observed. In experiments with p-String method, as $c$ increases ($\kappa$ fixed), each of vertical fixed point eventually bifurcates into a periodic orbit. Figure \ref{fig_PO_vertical} illustrates one of such periodic orbits, which bifurcated from the 1-nucleation fixed point in Figure \ref{fig_CP_vertical}. A video of this periodic orbit is available at \url{http://youtu.be/rJ74090jIvI} .
% note this periodic orbit was much better than the ones we found by projection, in the sense that near periodic dynamics persist much longer

Figure \ref{fig_PO_cBifurcation} illustrates numerically when bifurcation occurs to the 1-nucleation fixed point. Such bifurcations are intuitive, because when $c$ is small, fixed points are suggested by Proposition \ref{thm_verticalFixPtAsymptotics}, but when $c$ is large, the system is dominated by shear that leads to periodic dynamics.

\begin{figure}[h]
    \centering
    \includegraphics[width=0.25\textwidth]{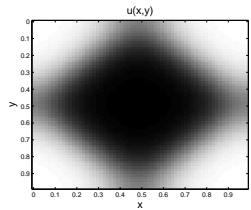}
    \caption{\footnotesize Diamond fixed point numerically obtained by p-String method. $\kappa=0.005$, $c=0$.}
    % n=40, h=1e-1, k=1/64, kappa=0.005, dynamics termination: only 1 step
    \label{fig_CP_diamond}
\end{figure}

Lastly, note there are fixed points other than horizontal and vertical types. For example, Figure \ref{fig_CP_diamond} illustrates another fixed point at $c=0$, obtained by p-String method with radial initial path (see Appendix \ref{sec_initialPath}). %When $c\neq 0$, albeit small, p-String method with the same initial path leads to a horizontal fixed point; this is unlike the case for vertical fixed points, which persist at small $c$. We conjecture that this diamond fixed point bifurcates at $c=0$ in the bifurcation diagram.

\paragraph{Local MLPs.}
Consider paths from $u=-1$ to $u=1$. For conciseness, we only describe transitions through (i) uniform $u=0$, (ii) 1-nucleation horizontal fixed point, and (iii) periodic orbit that bifurcated from 1-nucleation vertical fixed point.

\smallskip
(i) The simplest transition is through uniform $u=0$. Provably, there is a local MLP that contains only uniform images (i.e., independent of $x$ or $y$), and its action is the same as the minimum action of transition from $z=-1$ to $z=1$ in $\dot{z}=z-z^3$, which is 0.5 according to Theorem \ref{thm_action}. Both gMAM and up-down gMAM produced accurate approximations of this local MLP (detailed results not shown).

\smallskip

\begin{figure}[h]
    \centering
    \footnotesize
    \subfigure[Shear-indifferent MLP, obtained by up-down gMAM with linear initial path.]{
    \includegraphics[width=\textwidth]{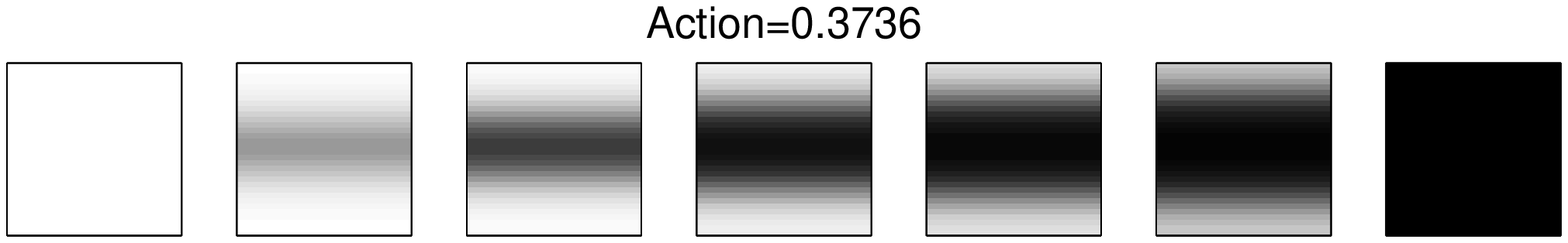}
    }
    \subfigure[Shear-facilitated MLP, obtained by up-down gMAM with elliptical initial path.]{
    \includegraphics[width=\textwidth]{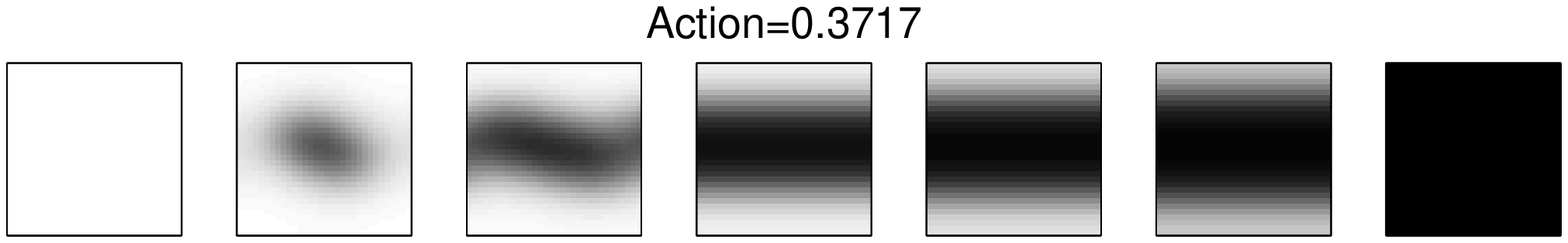}
    }
    \caption{\footnotesize Local MLPs through 1-nucleation horizontal fixed point. Each local MLP is illustrated by seven snapshots, uniformly distributed from reparameterized time 0 to 1; the middle snapshot corresponds to the fixed point. $\kappa=0.01$, $c=0.1$.}
    % n1=100, n2=10, TolErr=1e-4; k=1/32, kappa=0.01, c=0.1
    \label{fig_gMAMextend_k32_hor}
\end{figure}

(ii) Numerically we found multiple local MLPs through the same 1-nucleation horizontal fixed point. The fixed point was computed by p-String method (Section \ref{sec_POid}; with $\text{threshold}=10^{-6}$) using horizontal initial path. Figure \ref{fig_gMAMextend_k32_hor} illustrates two local MLPs computed by up-down gMAM. For the first local MLP, each point on the path is independent of $x$, and the shear plays no role in the transition. This case is essentially the same as Allen-Cahn in 1D space (Section \ref{sec_AC1D}), and the true minimum action can be computed by Theorem \ref{thm_action} as $\approx 0.3732$. For the second local MLP (first documented in \cite{HeVa08c}), however, shear facilitates the transition in the sense that action is smaller (and hence the transition is more likely).

gMAM with appropriate initial paths reproduce local MLPs of both types. Resulting paths are not visually discernable from up-down gMAM's results, and therefore not shown. gMAM minimum actions are $\approx 0.3740$ and $\approx 0.3720$, slightly less optimized than that of up-down gMAM. The local MLP seems to cross the fixed point, as we computed the minimum of $L^2$ distances between the fixed point and each image on gMAM path to be $\approx 0.0049 \ll 1$ in the shear-indifferent case, and $\approx 0.0093 \ll 1$ in the shear-facilitated case.

We thus conclude the system cannot be of orthogonal-type, because at least two local MLPs with different action values cross the separatrix at the same fixed point; otherwise there will be a contradiction with Theorem \ref{thm_action}.

This finding does not contradict the definition of quasipotential, because the quasipotential is a global infimum and thus unique, but what we observed are local minimizers and there could be many of them. For orthogonal-type systems, the action local minimum was proved to be unique once $x_a$ and $x_s$ are given, but now we see numerically it is not always the case.

% gMAM uses n=100, TolErr=1e-4; k=1/32, kappa=0.01, c=0.1

\smallskip

\begin{figure}[h]
    \centering
    \footnotesize
    \subfigure[Shear-hindered MLP, obtained by up-down gMAM with linear initial path; intersection with the separatrix was given by p-String method.]{
    \includegraphics[width=\textwidth]{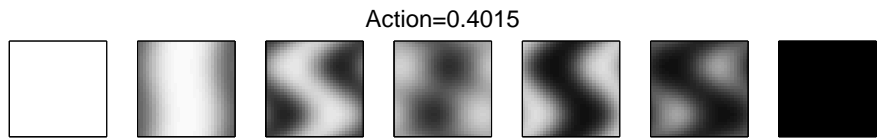}
    \label{fig_gMAMextend_k32_ver_up_down_gMAM}
    }
    \subfigure[Shear-hindered MLP, obtained by gMAM with vertical initial path.]{
    \includegraphics[width=\textwidth]{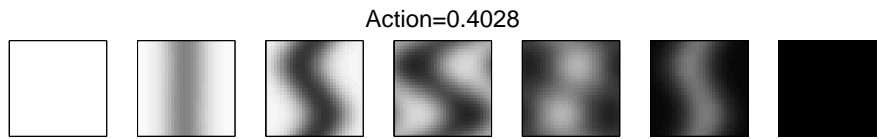}
    \label{fig_gMAMextend_k32_ver_gMAM}
    }
    \subfigure[Shear-hindered MLP, obtained by up-down gMAM with linear initial path; intersection with the separatrix was given by the point on the periodic orbit (obtained by p-String method) closest to the gMAM MLP above.]{
    \includegraphics[width=\textwidth]{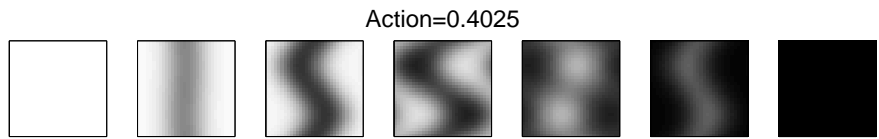}
    \label{fig_gMAMextend_k32_ver_closest2gMAM}
    }
    \caption{\footnotesize Local MLPs through the periodic orbit bifurcated from 1-nucleation vertical fixed point. Each local MLP is illustrated by seven snapshots, uniformly distributed from reparameterized time 0 to 1; the middle snapshot corresponds to crossing point on the periodic orbit. $\kappa=0.01$, $c=0.1$.}
    % n1=100, n2=10, n=100, TolErr=1e-4; k=1/32, kappa=0.01, c=0.1
    \label{fig_gMAMextend_k32_ver}
\end{figure}

(iii) There seems, like the orthogonal case, that each point on the periodic orbit (bifurcated from 1-nucleation vertical fixed point) is associated with at least one local MLP. See Figure \ref{fig_gMAMextend_k32_ver} for several local MLPs computed by gMAM and up-down gMAM. Figure \ref{fig_gMAMextend_k32_ver_up_down_gMAM} uses $x_s$ computed by p-String method (Section \ref{sec_POid}; with $\text{threshold}=10^{-6}$ and vertical initial path). This $x_s$ is not where the gMAM result crosses the separatrix (Figure \ref{fig_gMAMextend_k32_ver_gMAM}); however, both points are approximately on the same periodic orbit (the $L^2$-induced distance between gMAM-approximated MLP and the periodic orbit is $\approx 0.0199 \ll 1$). An additional up-down gMAM simulation with its separatrix-crossing aligned to that of gMAM (Figure \ref{fig_gMAMextend_k32_ver_closest2gMAM}) produces a path visually identical to that by gMAM. Meanwhile, up-down gMAM produces slightly better optimized action values (all three actions will be equal if there were infinite computing power).

This class of local MLPs are certainly not the global MLP because of their larger action values --- shear actually hinders the transition in these cases.

Note the gMAM result alone (Figure \ref{fig_gMAMextend_k32_ver_gMAM}) is not sufficient to demonstrate crossing at a periodic orbit: we evolved points on gMAM-approximated MLP using pure dynamics, and their evolutions did not show much periodic behavior. This is because gMAM only coarsely approximate a point on the periodic orbit (see the discussion on gMAM in Section \ref{sec_examples_3D} and Figure \ref{fig_3DgMAMoriginal}).

(iv)
Finally, the four types of local MLPs obtained above are compared in Figure \ref{fig_actionVsC} in terms of action values. There is an optimal shear strength $c$ that achieves the most likely transition among them.

\begin{figure}[h]
    \centering
    \includegraphics[width=0.5\textwidth]{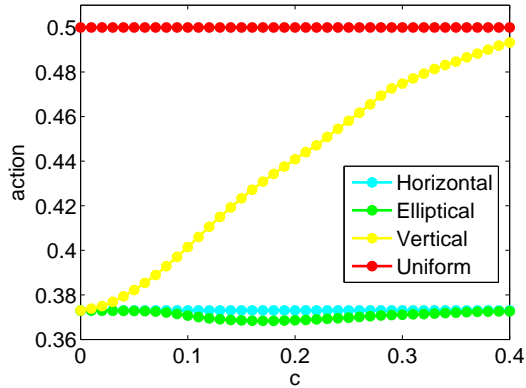}
    \caption{\footnotesize Action values of MLPs numerically obtained by gMAM as functions of $c$. Four types of MLPs are: `Horizontal' -- above case (ii)-1: through horizontal fixed point, invariant in $y$, indifferent to shear; `Elliptical' -- above case (ii)-2: through horizontal fixed point, facilitated by shear; `Vertical' -- above case (iii): through vertical fixed point or its bifurcated periodic orbit, hindered by shear; `Uniform' -- above case (i): through $u=0$ fixed point, indifferent to shear. $\kappa=0.01$ and $c$ samples from $0$ to $0.4$ with an increment $0.01$.}
    % n=40, h=0.01, k=1/32, kappa=0.01
    \label{fig_actionVsC}
\end{figure}

\section{Conclusion}
This article studies how metastable transitions in nongradient systems could differ from those in gradient systems, by investigating MLPs that minimize Freidlin-Wentzell action functional. In nongradient systems, there could be hyperbolic periodic orbits that are attracting on the separtrix. For a class of nongradient systems called orthogonal-type, it is demonstrated theoretically and by finite- and infinite-dimensional examples, that each such periodic orbit is associated with infinitely many local minimizers of the action functional, all of which have the same action characterized by a barrier height, and the corresponding local MLPs are the concatenations of two infinite-length heteroclinic orbits in two different deterministic dynamical systems. We also provided a non-orthogonal nongradient example, in which local MLPs through hyperbolic periodic orbit were numerically obtained. What contrasts the orthogonal case is the numerical observation of multiple local action minima associated with a single separatrix crossing location. Unfortunately, a theory for non-orthogonal systems is still incomplete; our argument only suggests that there is an MLP that crosses a hyperbolic attractor on the separatrix, but it remains unproved whether any hyperbolic attractor on the separatrix corresponds to at least one local MLP.

Two numerical methods were proposed and used in these investigations. One is a variant of String method named p-String method, which identifies hyperbolic periodic orbits in general deterministic dynamical systems. The other is up-down gMAM method, which improves gMAM in terms of both accuracy and efficiency by utilizing an input of separatrix crossing location from, for example, the p-String method.

\section{Appendix}
\subsection{Brief review of String method}
\label{sec_review_String}
Consider a system
\[
    dX=-\nabla V(X) dt +\sqrt{\epsilon} dW,
\]
and two local minima of $V$, $x_a$ and $x_b$. A Minimum Energy Path (MEP) between $x_a$ and $x_b$ is a curve $\phi(\alpha)$ parametrized by $\alpha$ that connects them and satisfies
\[
    (\nabla V)^{\perp}(\phi(\alpha))=0, \forall \alpha,
\]
where $(\nabla V)^{\perp}$ is the projection of $\nabla V$ orthogonal to $\phi$, i.e.,
\[
    (\nabla V)^{\perp}(\phi(\alpha)):=\nabla V(\phi(\alpha))-\left\langle \nabla V(\phi(\alpha)),\hat{\tau}(\alpha) \right\rangle \hat{\tau}(\alpha),
\]
with $\hat{\tau}(\alpha)=\phi_\alpha(\alpha)/\|\phi_\alpha(\alpha)\|$ being the unit tangent of $\phi$.

String method \cite{StringMethod} approximates an MEP by evolving a string $\phi(\alpha,t)$ in a fictitious time $t$, according to
\[
    \phi_t=-\nabla V(\phi)^\perp + \lambda \hat{\tau},
\]
where $\lambda=\lambda(\alpha,t)$ is a Lagrange multiplier that ensures a constant distance parametrization $\|\phi_\alpha\|_\alpha=0$. Simplified String method \cite{StringSimplified} further simplifies the dynamics to
\[
    \phi_t=-\nabla V(\phi) + r \hat{\tau},
\]
where $r$ again ensures constant distance.

Numerically, the string is discretized to $n+1$ points $\phi_i(t)$ and evolved by a splitting scheme based on alternating two substeps: at each step, first each discrete point is evolved by the same timestep using $\dot{\phi}_i=-\nabla V(\phi_i)$, and then reparametrization is implemented by redistributing points along the string via an interpolation. After a numerically converged evolution, the string at the final step approximates an MEP.

\subsection{Brief review of geometrized Minimum Action Method (gMAM)}
\label{sec_review_gMAM}
gMAM \cite{HeVa08a, VaHe08b, HeVa08c} established Lemma \ref{thm_geometricAction} and proposed to seek local MLP by optimizing the geometric action via a preconditioned steepest-descent algorithm, which evolves a path $X$ in a fictitious time $t$ according to
\[
    X_t=-\lambda \frac{\delta \hat{S}}{\delta X}, \quad \text{with} \quad \lambda:=\|f\|/\|X'\|.
\]
Due to the fact that $\hat{S}$ does not depend on $X$'s parameterization, a constant distance parameterization (i.e. $\|X'\|'=0$) is maintained so that the gradient descent dynamics remain well-conditioned.

The gradient can be computed by calculus of variations as
\[
    \frac{\delta \hat{S}}{\delta X}= - \lambda X''+(\nabla f-\nabla f^T)X' + \lambda^{-1} (\nabla f)^T f - \lambda' X'
\]
To numerically simulate the gradient flow, gMAM alternates between substeps of evolution and interpolation, the latter for ensuring the constant distance parameterization. Details can be found in \cite{HeVa08c}. The same idea applies to SPDEs (see also \cite{HeVa08a}). We write the two SPDE examples considered in this article (Sections \ref{sec_AC1D} and \ref{sec_AC2D_intro}) in a general (2+1)D form as
\[
    \partial_t \phi=\kappa\Delta\phi+f(\phi)+g(y)\partial_x\phi+\sqrt{\epsilon}\eta.
\]
Its geometrized action is
\begin{eqnarray}
    \hat{S}[\phi] & =\int_0^1 \Big( \sqrt{\int_{\mathbb{T}^2} | \phi'(s,z) |^2 dz} \sqrt{\int_{\mathbb{T}^2} | \kappa \Delta \phi(s,z) +f(\phi(s,z))+g(y)\partial_x\phi |^2 dz} \nonumber\\
    & -\int_{\mathbb{T}^2} \phi'(s,z) \big( \kappa \Delta \phi(s,z) +f(\phi(s,z))+g(y)\partial_x\phi \big) dz \Big) ds ,
    \label{eq_geometricAction}
\end{eqnarray}
where $z=(x,y)\in \mathbb{T}^2$ is the space coordinate and prime indicates partial derivative with respect to the reparametrized time $s$.

Calculus of variations computations lead to
\begin{eqnarray*}
    \frac{\delta \hat{S}}{\delta \phi} &= -\lambda' \phi' - \lambda \phi'' + \frac{1}{\lambda} \big( \kappa^2 \Delta\Delta\phi + 2\kappa f' \Delta\phi + \kappa f'' \nabla\phi\cdot\nabla\phi + f f' \nonumber\\
    & + \kappa g''(y) \partial_x\phi+\kappa 2g'(y)\partial_{xy}\phi-g(y)^2\partial_{xx}\phi \big)+2g(y)\partial_x\phi' ,
\end{eqnarray*}
where $f'$ and $f''$ denote $\partial f(\phi)/\partial \phi$ and $\partial^2 f(\phi)/\partial \phi^2$ (here prime on $f$ does not mean time derivative), and $\lambda(s)$ is defined by
\[
    \lambda=\frac{\sqrt{\int_{\mathbb{T}^2} | \kappa \Delta \phi(s,z) +f(\phi(s,z))+g(y)\partial_x\phi |^2 dz}}{\sqrt{\int_{\mathbb{T}^2} | \phi'(s,z) |^2 dz}}.
\]
The minimization is again performed by preconditioned gradient descent dynamics
\begin{align*}
    \phi_t &= -\lambda \frac{\delta \hat{S}}{\delta \phi} \\
    &= \lambda \lambda' \phi' + \lambda^2 \phi'' - \kappa^2 \Delta\Delta\phi - 2\kappa f' \Delta\phi - \kappa f'' \nabla\phi\cdot\nabla\phi - f f' \\
    & ~~ - \kappa g''(y) \partial_x\phi - \kappa 2g'(y)\partial_{xy}\phi + g(y)^2\partial_{xx}\phi - \lambda 2g(y)\partial_x\phi'
\end{align*}
which is a PDE in 4-dimension (fictitious time $t$ for optimization, reparametrized physical time $s$, and space $x$ and $y$). To numerical evolve this dynamics, we use pseudospectral discretization for $x$ and $y$, and 2nd-order central difference for $s$ (1st-order at boundaries). Time stepping in $t$ is done by Strang splitting, where in the first and third substeps the $- \kappa^2 \Delta\Delta\phi$ term is integrated by an exponential solver for half step, and the second substep is a full step Crank-Nicolson for the remaining system, where the $\lambda^2 \phi''$ term is diffusion-like and treated implicitly, and the rest terms are time-stepped explicitly.

\subsection{Properties of the 1D-space SPDE}
\label{sec_SPDE1D_structures}

\begin{proof}[Proof of Proposition \ref{thm_AC1D_orthogonal}]
    Simple calculus of variations and integration by part using periodic boundary condition shows
    \[
        \delta V=\int_0^1 \left( -\kappa u_{xx} \delta u-u(1-u^2)\delta u \right) dx
    \]
    and
    \begin{align*}
        \left\langle \frac{\delta V}{\delta u}, b \right\rangle
        &= \int_0^1 - (\kappa u_{xx}+u-u^3) c u_x \, dx \\
        &= -c\int_0^1 \kappa \left( \frac{u_x^2}{2} \right)_x+ \left( \frac{u^2}{2} \right)_x- \left( \frac{u^4}{4} \right)_x \, dx
        = 0 .
    \end{align*}
\end{proof}

Then we show that $V[\phi(\cdot,t)]$ is a Lyapunov function of the system without noise (note this proof extends to any orthogonal-type nongradient system):
\begin{Corollary}
    Given a solution $\phi(x,t)$ of \eqref{eq_AC1D}, we have
    \[
        \frac{dV[\phi(\cdot, t)]}{dt} \leq 0
    \]
    Equality occurs if and only if $\phi$ satisfies $\kappa \phi_{xx}+\phi-\phi^3=0$.
    \label{thm_Lyapunov}
\end{Corollary}

\begin{proof}
    By chain rule
    \begin{align*}
        \frac{dV[\phi(\cdot, t)]}{dt} = \int_0^1 \frac{\delta V}{\delta \phi} \phi_t \, dx
        = \left\langle \frac{\delta V}{\delta \phi}, -\frac{\delta V}{\delta \phi}+b \right\rangle
        = -\left\| \frac{\delta V}{\delta \phi} \right\|^2 \leq 0
    \end{align*}
    Inequality becomes equality if and only if $\frac{\delta V}{\delta \phi}=0$, i.e., $\kappa \phi_{xx}+\phi-\phi^3=0$.
\end{proof}

Now we prove the statements on fixed points and periodic orbits.

\begin{proof}[Proof of Proposition \ref{thm_AC1D_fixedPt1} and \ref{thm_AC1D_fixedPt2}]
    If $u(x)$ is a fixed point of \eqref{eq_AC1D}, it needs to satisfy
    \begin{equation}
        \kappa u_{xx}+u-u^3+c u_x=0
        \label{eq_subSystem}
    \end{equation}
    with boundary condition $u(x)=u(x+1)$.

    Assume without loss of generality that $c\geq 0$, because there is a 1-to-1 correspondence between solutions of \eqref{eq_subSystem} with $c=c_0$ and $c=-c_0$ via a coordinate change $x\mapsto -x$.

    Writing $q=u$ and $p=dq/dx$, and letting $H(q,p)=\kappa p^2/2-(1-q^2)^2/4$, \eqref{eq_subSystem} can be rewritten as
    \[ \begin{cases}
        q_x &=p \\
        \kappa p_x &=-\partial H/\partial q-cp
    \end{cases} \]
    and recognized as a mechanical system with dissipation.

    There are three critical points of the potential energy $U(q)=-(1-q^2)^2/4$, namely $q=0$, $q=1$ and $q=-1$. Clearly, $u_s(x)=0$, $u_+(x)=-1$ and $u_-(x)=1$ are solutions of \eqref{eq_subSystem}, and they trivially satisfy the boundary condition.

    $u_+=1$ and $u_=-1$ are stable, because they give zero value to the non-negative Lyapunov function $V[\cdot]$, and Corollary \ref{thm_Lyapunov} shows they are the only global minimizers. On the other hand, $u_s=0$ is unstable, because there are homogeneous states $u(x)= \epsilon$ and $u(x)= -\epsilon$ in its arbitrarily small neighborhood that correspond to smaller $V$ values.

    \begin{figure}[h]
        \centering
        \vspace{-10pt}
        \includegraphics[width=0.5\textwidth]{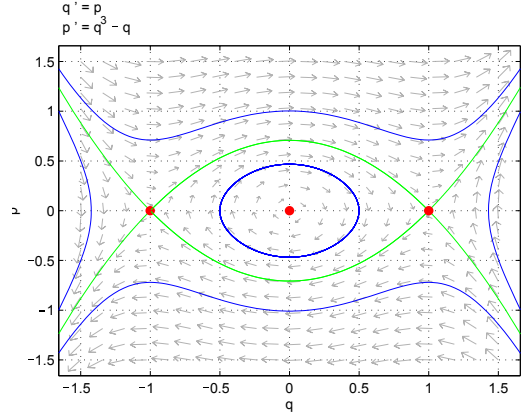}
        \vspace{-30pt}
        \caption{\footnotesize Phase portrait of Hamiltonian dynamics \eqref{eq_CP_Hamiltonian} with $\kappa=1$.}
        \label{fig_CP_Hamiltonian}
    \end{figure}

    When $c>0$, \eqref{eq_subSystem} is dissipative and $H(q,p)$ always converges to a local minimum of the potential energy $U$. Therefore, the only solutions $u(x)$ that satisfy the periodic boundary condition are the constant solutions $u_s$, $u_+$ and $u_-$. They give the only three fixed points in \eqref{eq_AC1D}.

    When $c=0$, the fixed point satisfies
    \begin{equation}
        \kappa u_{xx}+u-u^3=0,  \quad u(x)=u(x+1)
        \label{eq_CP_Hamiltonian}.
    \end{equation}
    This can be viewed as a 1D nonlinear oscillator, whose phase portrait is illustrated in Figure \ref{fig_CP_Hamiltonian}. We now count solutions that satisfy the boundary condition, which means they have to be either constant or periodic orbits with period $1/N:N \in \mathbb{Z}^+$.

    The energy in \eqref{eq_CP_Hamiltonian} is obviously conserved along any trajectory, and therefore we can let $E=H(q(t),p(t))$. It is easy to see only $E\in[-1/4,0]$ corresponds to a closed orbit.

    Rewrite $dq/dx=p$ as
    \[
        dx=\frac{1}{p} dq=\frac{\sqrt{\kappa}}{\pm \sqrt{2E+(1-q^2)^2/2}} dq.
    \]
    As can be seen in the phase portrait, a closed orbit at energy $E$ first goes from $(q_l,0)$ to $(q_r,0)$ and then goes back, where $q_l<0$, $q_r>0$, and $E=-(1-q_l^2)^2/4=-(1-q_r^2)^2/4$. The period of this orbit is
    \[
        \Delta x =2 \int_{-\sqrt{1-2\sqrt{-E}}}^{\sqrt{1-2\sqrt{-E}}} \frac{\sqrt{\kappa}}{\sqrt{2E+(1-q^2)^2/2}} dq
    \]
    Although there is no closed form expression for this integral, it can be shown that $\Delta x$ continuously deceases as $E$ decreases from 0 to $-1/4$.

    Consider two extremes: $E=0$ and $E=-1/4$. When $E=0$, the `periodic orbit' is the union of two heteroclinic orbits linking $q=-1$ and $q=1$, and the `period' is $\Delta x=\infty$. When $E=-1/4$ (corresponding to $q=0,p=0$), the periodic orbit degenerates to the fixed point $q=0$. To study periodic orbits near this fixed point, consider initial condition in an $\epsilon$ neighborhood of the origin, which linearizes \eqref{eq_CP_Hamiltonian}, and the solution is approximately harmonic, i.e.,
    \[
        u= \epsilon \left( a\cos{\frac{x}{\sqrt{\kappa}}}+b \sin{\frac{x}{\sqrt{\kappa}}}\right) +o(\epsilon)
    \]
    Therefore, among all periodic orbits in \eqref{eq_CP_Hamiltonian}, the smallest period is $2\pi \sqrt{\kappa}$ near $q=p=0$, and the period increases to $\infty$ at the heteroclinic orbits.

    Since $[2\pi \sqrt{\kappa},\infty) \bigcap \{1,1/2,1/3,\cdots\}$ is a finite set, only finitely many of the solutions have period $1/N$. $2\pi\sqrt{\kappa} \leq 1$ is necessary for there to be at least one, and as $\kappa$ decreases, the total amount will be nondecreasing. When $\kappa> 1/(2\pi)^2$, the solutions to \eqref{eq_CP_Hamiltonian} that satisfy the boundary conditions are only constant $u=-1,0,1$.
\end{proof}

\begin{proof}[Proof of Proposition \ref{thm_AC1D_PO}]
    Note if $u(x)$ satisfies $\kappa u_{xx}+u-u^3=0$, then $\phi(x,t):=u(x+c t)$ solves $\phi_t=\kappa \phi_{xx}+\phi-\phi^3+c \phi_x$. This is because by chain rule,
    \[
        -\phi_t+\kappa\phi_{xx}+\phi-\phi^3+c\phi_x = -c u'(x+ct)+\kappa u''(x+ct)+u(x+ct)-u(x+ct)^3+c u'(x+ct) = 0.
    \]
    Since any non-constant fixed point of \eqref{eq_AC1D} with $c=0, \epsilon=0$ satisfies $\kappa u_{xx}+u-u^3=0$ and $u(x)=u(x+1)$, $\phi(x,t)$ solves \eqref{eq_AC1D} with $c\neq 0, \epsilon=0$ and satisfies $\phi(x,t)=\phi(x,t+1/c)$. Therefore, it is a periodic orbit with $1/|c|$ period in $t$.
\end{proof}

\subsection{The triviality of an orthogonal decomposition of $\dot{q}=p,\dot{p}=-q$}
\label{sec_orthoHamiltonian}
\begin{Proposition}
    If a scalar field $V$ and a vector field $b$ satisfies $-\nabla V(q,p)+b(q,p)=(p,-q)$ and $\nabla V \cdot b=0$, then $V(q,p) \equiv \text{constant}$.
\end{Proposition}
\begin{proof}
    Let $x=(q,p)$. It is easy to see that $V(x(t))$ is a Lyapunov function in
    \[
        \dot{x}=-\nabla V(x)+b(x).
    \]
    However, for any $r\geq 0$, $q^2+p^2=r^2$ is a periodic orbit in the system, and therefore $V$ must be constant on each of these periodic orbit. Hence there exists a scalar function $U(r)$ such that
    \[
        V(q,p)=U(r).
    \]
    Since $\nabla V \cdot b=0$ is equivalent to $\nabla V \cdot ((p,-q)+\nabla V)=0$, chain rule leads to
    \[
        U'(r) \begin{bmatrix} \frac{q}{r} \\ \frac{p}{r} \end{bmatrix} \cdot \begin{bmatrix} p+U'(r) \frac{q}{r} \\ -q+ U'(r) \frac{p}{r} \end{bmatrix} = 0,
    \]
    and therefore $[U'(r)]^2 r=0$. This leads to $U(r) \equiv \text{constant}$, and hence $V(q,p) \equiv \text{constant}$.
\end{proof}

\subsection{Initial paths used in path evolutions for the sheared Allen-Cahn system}
\label{sec_initialPath}
Denote by $n+1$ the number of points on a discrete path, with $\phi_j$ being the $j$-th point, $1 \leq j \leq n+1$. Here are initial paths used in our path evolutions.
\begin{itemize}
\item
    `Linear'. Given any two points $\phi_a$ and $\phi_b$, the path is given by
    \[
        \phi_j = \phi_a \frac{n+1-j}{n} + \phi_b \frac{j-1}{n}.
    \]
\item
    `Horizontal'. This path corresponds to a non-optimal nucleation in $y$ direction between $\phi_1(x,y)=-1$ and $\phi_{n+1}(x,y)=1$. Points on this path are given by Gaussian in $y$ with width controlled by $j$. More specifically,
    \[
        \phi_j(x,y)=2 \exp \left( -\frac{(0.5-y)^2}{4/9(j/n)^2} \right)-1
    \]
    for $2 \leq j \leq n$.
\item
    `Double horizontal'. This path corresponds to two nucleations in $y$ direction between $\phi_1(x,y)=-1$ and $\phi_{n+1}(x,y)=1$, obtained by stacking two `Horizontal'. More specifically,
    \[
        \phi_j(x,y)=2 \exp \left( -\frac{(0.5-2\cdot \text{mod}(y,0.5))^2}{4/9(j/n)^2} \right)-1
    \]
    for $2 \leq j \leq n$.
\item
    `Elliptical'. This path corresponds to skewed prolate Gaussians. More specifically,
    \[
        \phi_j(x,y)=2 \exp \left( -\frac{(y-x/16-15/32)^2+(y/16+x-17/32)^2/16}{4/9(j/n)^2} \right)-1
    \]
    for $2 \leq j \leq n$.
\item
    `Vertical'. `Horizontal' with $x$ and $y$ swapped.
\item
    `Double vertical'. `Vertical' with $x$ and $y$ swapped.
\item
    `Radial'. This path corresponds to a non-optimal nucleation in $\sqrt{x^2+y^2}$ direction between $\phi_1(x,y)=-1$ and $\phi_{n+1}(x,y)=1$. More specifically,
    \[
        \phi_j(x,y)=2 \exp \left( -\frac{(0.5-x)^2+(0.5-y)^2}{4/9(j/n)^2} \right)-1
    \]
    for $2 \leq j \leq n$.
\end{itemize}
Note these specific function forms are not necessary, as long as symmetries are broken in the same way.

\section{Acknowledgment}
This work was supported by NSF grant DMS-1521667. The problem was set up by Eric Vanden-Eijnden, with whom the author gratefully had multiple stimulating discussions. We also thank Robert Kohn, John Guckenheimer, David S Cai, and two anonymous referees for helpful suggestions, which greatly improved the quality of this article.

\footnotesize
\bibliographystyle{siam}
\bibliography{molei25}

\end{document}